\documentclass[11pt,reqno]{amsart}
\usepackage{amsmath,amsfonts,amssymb,amsthm,bm,booktabs,color,setspace,tikz}
\usepackage[margin=1.25in]{geometry}
\usepackage[
	colorlinks=true, 
	pdfstartview=FitV, 
	linkcolor=blue, 
	citecolor=blue, 
	urlcolor=blue, 
	bookmarksdepth=2]{hyperref}
\usetikzlibrary{matrix}
\usepackage[vcentermath]{genyoungtabtikz}

\tikzset{tab/.style={matrix of math nodes,column sep=-.35, row sep=-.35,text height=7pt,text width=7pt,align=center,inner sep=2,font=\footnotesize}}

\tikzset{dynkdot/.style={circle,draw,scale=.45}}
\tikzset{dot/.style={circle,draw,fill,scale=.45}}
\newcommand{\tikzmark}[2]{\tikz[overlay,remember picture,baseline] \node [anchor=base] (#1) {$#2$};}

\usetikzlibrary{decorations.pathreplacing}

\makeatletter
\def\@tocline#1#2#3#4#5#6#7{\relax
	\ifnum #1>\c@tocdepth 
	\else
	\par \addpenalty\@secpenalty\addvspace{#2}%
	\begingroup \hyphenpenalty\@M
	\@ifempty{#4}{%
		\@tempdima\csname r@tocindent\number#1\endcsname\relax
	}{%
	\@tempdima#4\relax
}%
\parindent\z@ \leftskip#3\relax \advance\leftskip\@tempdima\relax
\rightskip\@pnumwidth plus4em \parfillskip-\@pnumwidth
#5\leavevmode\hskip-\@tempdima
\ifcase #1
\or\or \hskip 1em \or \hskip 2em \else \hskip 3em \fi%
#6\nobreak\relax
\dotfill\hbox to\@pnumwidth{\@tocpagenum{#7}}\par
\nobreak
\endgroup
\fi}
\makeatother

\newcommand{\mybox}[1]{
 \begin{tikzpicture}[baseline=4,scale=.5]
  \draw[-] (0,0) -- (1,0) -- (1,1) -- (0,1) -- cycle;
  \node[font=\footnotesize] at (.5,.5) {$#1$};
 \end{tikzpicture}}
\newcommand{\mylongbox}[1]{
 \begin{tikzpicture}[baseline=4,scale=.5]
  \draw[-] (0,0) -- (2,0) -- (2,1) -- (0,1) -- cycle;
  \node[font=\footnotesize] at (1,.5) {$#1$};
 \end{tikzpicture}}

\newcommand{\rw}[1]{
	\mathrm{read_{#1}}
	}

\newcommand{\arxiv}[1]{\href{http://arxiv.org/abs/#1}{\tt arXiv:\nolinkurl{#1}}}

\newcommand{\BR}{\mathrm{br}}
\newcommand{\bz}{\mathbf{Z}}

\newcommand{\g}{\mathfrak{g}}

\newcommand{\Kp}{\mathrm{Kp}}
\newcommand{\mdots}{\raisebox{1.5ex}{$\vdots$}}

\newcommand{\R}{\mathcal{R}}

\newcommand{\stack}[1]{\begin{smallmatrix}#1\end{smallmatrix}}
\newcommand{\TT}{\mathcal{T}}
\newcommand{\wt}{{\rm wt}}

\numberwithin{equation}{section}
\numberwithin{figure}{section}
\numberwithin{table}{section}
\newtheorem{Theorem}[equation]{Theorem}

\newtheorem{Proposition}[equation]{Proposition} 
\newtheorem{Lemma}[equation]{Lemma}

\theoremstyle{definition}

\newtheorem{Definition}[equation]{Definition}
\newtheorem{Example}[equation]{Example}
\newtheorem{Remark}[equation]{Remark}

\begin{document}
\title[PBW bases and tableaux]{PBW bases and marginally large tableaux in types B and C}

\author{Jackson Criswell}
\address{Department of Mathematics, Central Michigan University, Mount Pleasant, MI}
\email{crisw1ja@cmich.edu}
\author{Ben Salisbury}
\thanks{B.S.\ was partially supported by Simons Foundation grant \#429950}
\address{Department of Mathematics, Central Michigan University, Mount Pleasant, MI}
\email{ben.salisbury@cmich.edu}
\urladdr{http://people.cst.cmich.edu/salis1bt/}
\author{Peter Tingley}
\thanks{P.T.\ was partially supported by NSF grant DMS-1265555}
\address{Department of Mathematics and Statistics, Loyola University, Chicago, IL}
\email{ptingley@luc.edu}
\urladdr{http://webpages.math.luc.edu/~ptingley/}
\subjclass[2010]{05E10,17B37}

\maketitle

\begin{abstract} 


We explicitly describe the isomorphism between two combinatorial realizations of Kashiwara's infinity crystal in types B and C. The first realization is in terms of marginally large tableaux and the other is in terms of Kostant partitions coming from PBW bases. We also discuss a stack notation for Kostant partitions which simplifies that realization. 

\end{abstract}

\setcounter{tocdepth}{1} 
\tableofcontents

\section{Introduction}\label{sec:intro} 
The infinity crystal $B(\infty)$ is a combinatorial object associated with a symmetrizable Kac--Moody algebra $\g$.  It contains information about the integrable highest weight representations of $\g$ and the associate quantum group $U_q(\g)$. Kashiwara's original description of $B(\infty)$ used a complicated algebraic construction, but there are often simple combinatorial realizations.  Here we consider two such realizations in types $B_n$ and $C_n$.  The first is the marginally large tableaux construction of \cite{C98,HL08}. The second uses the Kostant partitions from \cite{SST1}, which are related to Lusztig's PBW bases \cite{L10} (see also \cite{Tin}).  In \cite{CT15} and \cite{SST2}, isomorphisms between these two realizations are studied in types $A_n$ and $D_n$, respectively. 
Our main result is a simple description of the unique isomorphism between these two realizations of $B(\infty)$ for types $B_n$ and $C_n$. We also give a stack notation for Kostant partitions of these types motivated by the connection to multisegments in type $A_n$ described in \cite{CT15}.

\section{Background}
Let $\g$ be a Lie algebra of type $B_n$ or $C_n$. The Cartan matrix and Dynkin diagram are
\[
B_n:  (a_{ij}) =
\left(\begin{smallmatrix}
2 &-1 & 0 & \cdots & 0 & 0 & 0 \\
-1& 2 &-1 & \cdots & 0 & 0 & 0 \\
0 &-1 & 2 & \cdots & 0 & 0 & 0 \\
& & & \ddots & & & \\
0 & 0 & 0 & \cdots & 2 &-1 &0 \\
0 & 0 & 0 & \cdots & -1& 2 & -1 \\
0 & 0 & 0 & \cdots & 0& -2 & 2 \\ 
\end{smallmatrix}
\right), 
\qquad
C_n:  (a_{ij}) =
\left(\begin{smallmatrix}
2 &-1 & 0 & \cdots & 0 & 0 & 0 \\
-1& 2 &-1 & \cdots & 0 & 0 & 0 \\
0 &-1 & 2 & \cdots & 0 & 0 & 0 \\
& & & \ddots & & & \\
0 & 0 & 0 & \cdots & 2 &-1 &0 \\
0 & 0 & 0 & \cdots & -1& 2 & -2 \\
0 & 0 & 0 & \cdots & 0& -1 & 2 \\ 
\end{smallmatrix}
\right)
\]

\[
B_n: \ 
\begin{tikzpicture}[baseline=-5,scale=1,font=\scriptsize]
\foreach \x in {1,2,4,5}
{\node[dynkdot] (\x) at (\x,0) {};}
\node[label={below:$\alpha_1$}] at (1,0) {}; 
\node[label={below:$\alpha_2$}] at (2,0) {}; 
\node[label={below:$\alpha_{n-1}$}] at (4,0) {}; 
\node[label={below:$\alpha_n$}] at (5,0) {}; 
\node at (3,0) {$\cdots$};
\draw[-] (1.east) -- (2.west);
\draw[-] (2.east) -- (2.75,0);
\draw[-] (3.25,0) -- (4.west);
\draw[-] (4.30) -- (5.150);
\draw[-] (4.330) -- (5.210);
\draw[-] (4.55,0) -- (4.45,.1);
\draw[-] (4.55,0) -- (4.45,-.1);
\end{tikzpicture}
\qquad 
C_n: \
\begin{tikzpicture}[baseline=-5,scale=1,font=\scriptsize]
\foreach \x in {1,2,4,5}
{\node[dynkdot] (\x) at (\x,0) {};}
\node[label={below:$\alpha_1$}] at (1,0) {}; 
\node[label={below:$\alpha_2$}] at (2,0) {}; 
\node[label={below:$\alpha_{n-1}$}] at (4,0) {}; 
\node[label={below:$\alpha_n.$}] at (5,0) {}; 
\node at (3,0) {$\cdots$};
\draw[-] (1.east) -- (2.west);
\draw[-] (2.east) -- (2.75,0);
\draw[-] (3.25,0) -- (4.west);
\draw[-] (4.30) -- (5.150);
\draw[-] (4.330) -- (5.210);
\draw[-] (4.45,0) -- (4.55,.1);
\draw[-] (4.45,0) -- (4.55,-.1);
\end{tikzpicture}
\]

\begin{table}[t]
	\renewcommand{\arraystretch}{1.2}
	\[
	\begin{array}{cl}\toprule
	\beta_{i,k}= \alpha_i + \cdots + \alpha_{k}, & 1\le i\le k \le n \\
	\gamma_{i,k}=\alpha_i + \cdots + \alpha_{k-1}+ 2\alpha_{k} + 2\alpha_{k+1}+ \cdots + 2\alpha_n, & 1\le i < k \le n\\ \midrule
	
	\beta_{i,k}= \varepsilon_i-\varepsilon_{k+1}, & 1\le i\le k \le n-1 \\
	\beta_{i,n}= \varepsilon_i, &1\le i \le n\\
	\gamma_{i,k}= \varepsilon_i+\varepsilon_k, & 1\le i < k \le n\\\bottomrule
	\end{array}
	\]
	\caption{Positive roots of type $B_n$, expressed both as a linear combination of simple roots and in the canonical realization following \cite{bourbaki}.}
	\label{posrootsB}
\end{table}
\begin{table}[t] 
	\renewcommand{\arraystretch}{1.2}
	\[
	\begin{array}{cl}\toprule
	\beta_{i,k}= \alpha_i + \cdots + \alpha_{k}, & 1\le i\le k < n \\
	\gamma_{i,k}=\alpha_i + \cdots + \alpha_{n-1}+ \alpha_{n} + \alpha_{n-1}+ \cdots + \alpha_k, & 1\le i \le k \le n\\ \midrule
	
	\beta_{i,k}= \varepsilon_i-\varepsilon_{k+1}, & 1\le i\le k < n \\
	\gamma_{i,k}= \varepsilon_i+\varepsilon_k, & 1\le i \le k \le n\\\bottomrule
	\end{array}
	\]
	\caption{Positive roots of type $C_n$, expressed both as a linear combination of simple roots and in the canonical realization following \cite{bourbaki}.}
	\label{posrootsC}
\end{table}

Let $\{\alpha_1,\dots,\alpha_n\}$ be the simple roots and $\{\alpha_1^\vee,\dots,\alpha_n^\vee\}$ the simple coroots, related by the inner product $\langle\alpha_j^\vee,\alpha_i\rangle = a_{ij}$.  Define the fundamental weights $\{\omega_1,\dots,\omega_n\}$ by $\langle\alpha_i^\vee,\omega_j\rangle = \delta_{ij}$.  Then the weight lattice is $P = \bz\omega_1 \oplus \cdots \oplus \bz\omega_n$ and the coroot lattice is $P^\vee = \bz\alpha_1^\vee \oplus \cdots \oplus \bz\alpha_n^\vee$.  
Let $\Phi$ denote the roots associated to $\g$, with the set of positive roots denoted $\Phi^+$.  The list of positive roots in type $B_n$ is given in Table \ref{posrootsB}, and the list of positive roots in type $C_n$ is given in Table \ref{posrootsC}.  The Weyl group associated to $\g$ is the group generated by $s_1,\dots,s_n$, where $s_i(\lambda) = \lambda - \langle\alpha_i^\vee,\lambda\rangle\alpha_i$ for all $\lambda \in P$.  There exists a unique longest element of $W$ which is denoted as $w_0$.  For notational brevity, set $I = \{1,2,\dots,n\}$.

Let $B(\infty)$ be the infinity crystal associated to $\g$ as defined in \cite{K91}. This is a countable set along with operators $e_i$ and $f_i$, which roughly correspond to the Chevalley generators of $\g$. 
Here we use two explicit realizations  of $B(\infty)$ but do not need the general definition.

\subsection{Crystal of marginally large tableaux}\label{sec:MLT}
Recall the fundamental crystals given below. 
\begin{equation} \label{fundBC}
\begin{aligned}
B_n:  & \quad
\begin{tikzpicture}[baseline=-4,xscale=1.25]
	\node (1) at (0,0) {$\mybox1$};
	\node (d1) at (1.5,0) {$\cdots$};
	\node (n) at (3,0) {$\mybox{n}$};
	\node (z) at (4.5,0) {$\mybox{0}$};
	\node (bn) at (6,0) {$\mybox{\overline{n}}$};
	\node (d2) at (7.5,0) {$\cdots$};
	\node (b1) at (9,0) {$\mybox{\overline{1}}$};
	\draw[->] (1) to node[above]{\tiny$1$} (d1);
	\draw[->] (d1) to node[above]{\tiny$n-1$} (n);
	\draw[->] (n) to node[above]{\tiny$n$} (z);
	\draw[->] (z) to node[above]{\tiny$n$} (bn);
	\draw[->] (bn) to node[above]{\tiny$n-1$} (d2);
	\draw[->] (d2) to node[above]{\tiny$1$} (b1);
\end{tikzpicture} \\
C_n: &  \quad
\begin{tikzpicture}[baseline=-4,xscale=1.25]
	\node (1) at (0,0) {$\mybox1$};
	\node (d1) at (1.5,0) {$\cdots$};
	\node (n) at (3,0) {$\mybox{n}$};
	\node (bn) at (4.5,0) {$\mybox{\overline{n}}$};
	\node (d2) at (6,0) {$\cdots$};
	\node (b1) at (7.5,0) {$\mybox{\overline{1}}$};
	\draw[->] (1) to node[above]{\tiny$1$} (d1);
	\draw[->] (d1) to node[above]{\tiny$n-1$} (n);
	\draw[->] (n) to node[above]{\tiny$n$} (bn);
	\draw[->] (bn) to node[above]{\tiny$n-1$} (d2);
	\draw[->] (d2) to node[above]{\tiny$1$} (b1);
\end{tikzpicture}
\end{aligned}
\end{equation}
Define alphabets, denoted $J(B_n)$ and $J(C_n)$, to be the elements of these crystals with the natural orderings
\begin{align*}
J(B_n) &: \ \  \left\{ 1 \prec \cdots \prec n-1 \prec  n \prec 0 \prec \overline n  \prec \overline{n-1} \prec \cdots \prec \overline 1\right\}\text{, and}\\
J(C_n) &: \ \ \left\{ 1 \prec \cdots \prec n-1 \prec  n \prec \overline n  \prec \overline{n-1} \prec \cdots \prec \overline 1\right\}.
\end{align*}

\begin{Definition} \label{defMLT}
The set of marginally large tableaux, $\TT(\infty)$, is the set of semistandard Young tableaux $T$ with entries in $J(B_n)$ or $J(C_n)$ which satisfy the following conditions.
\begin{enumerate}
 \item The number of $\mybox{i}$ in the $i$-th row of $T$ is exactly one more than the total number of boxes in the $(i+1)$-th row.   
 \item Entries weakly increase along rows.
 \item All entries in the $i$-th row are $\preceq \overline{\imath}$.
 \item If $T$ is of type $B_n$, then the $\mybox{0}$ does not appear more than once per row. 
\end{enumerate}
\end{Definition}

Definition \ref{defMLT} implies that the leftmost column of $T$ contains $\mybox{1}\,$, $\mybox{2}\,$, \dots, $\mylongbox{n-1}\,$, $\mybox{n}$ in increasing order from top to bottom.  We call the $\mybox{i}$ in row $i$ \textit{shaded boxes}. The number of shaded boxes in each row is one more than the total number of boxes in the next row.  

\begin{Example} \label{ex:tableauxB}
In type $B_3$, each $T\in\TT(\infty)$ has the form
\[
T = 
\begin{tikzpicture}[baseline,font=\footnotesize]
\matrix [matrix of math nodes,column sep=-.4, row sep=-.5,text height=9pt,align=center,inner sep=3] 
	{
		\node[draw,fill=gray!30]{$1$}; & 
		\node[draw,fill=gray!30]{$1$}; & 
		\node[draw,fill=gray!30]{$1 \cdots 1$}; & 
		\node[draw,fill=gray!30]{$1$}; & 
		\node[draw,fill=gray!30]{$1\cdots1$}; & 
		\node[draw,fill=gray!30]{$1$}; & 
		\node[draw,fill=gray!30]{$1\cdots1$}; & 
		\node[draw,fill=gray!30]{$1\cdots1$}; & 
		\node[draw,fill=gray!30]{$1$}; & 
		\node[draw]{$2\cdots 2$}; & 
		\node[draw]{$3\cdots 3$}; & 
		\node[draw]{$0$}; & 
		\node[draw]{$\overline 3 \cdots \overline 3$}; & 
		\node[draw]{$\overline 2\cdots \overline 2$}; & 
		\node[draw]{$\overline 1 \cdots \overline 1$};\\
		\node[draw,fill=gray!30]{$2$}; & 
		\node[draw,fill=gray!30]{$2$}; & 
		\node[draw,fill=gray!30]{$2\cdots 2$}; & 
		\node[draw,fill=gray!30]{$2$}; & 
		\node[draw]{$3 \cdots 3$}; & 
		\node[draw]{$0$}; & 
		\node[draw]{$\overline 3\cdots \overline 3$}; &
		\node[draw]{$\overline 2\cdots \overline 2$}; \\
		\node[draw,fill=gray!30]{$3$}; & 
		\node[draw]{$0$}; & 
		\node[draw]{$\overline 3 \cdots \overline 3$}; \\
	};
\end{tikzpicture}.
\]
The notation $\mylongbox{i\cdots i}$ indicates any number of $\mybox{i}$ (possibly zero).  Also, the $\mybox{0}$ in each row may or may not be present.     
\end{Example}

%

\begin{Definition}\label{def:Tweight}
Fix $T\in\TT(\infty)$ for $1 \leq j \leq n-1$ and $k \succ j \in J$. Let $\mybox{k}_j$ denote a box containing $k$ in row $j$ of $T$.  Define the weight of the box by:
\begin{align*}
\text{Type } B_n: \quad \wt\left(\mybox{k}_j\right) &= 
\left\{ \begin{array}{cl}
-\beta_{j,k-1} & \text{if $k\neq0$},\\
-\beta_{j,n} & \text{if }k=0,
\end{array} \right. &
\wt\left(\mybox{\overline{k}}_j\right) &= 
\left\{ \begin{array}{cl}
-\gamma_{j,k} & \text{if $k\neq j$},\\
-2\beta_{j,n} & \text{if $k = j$.}
\end{array} \right.
\end{align*}
\begin{align*}
\text{Type } C_n: \quad \wt\left(\mybox{k}_j\right)&=-\beta_{j,k-1}, &
\wt\left(\mybox{\overline{k}}_j\right)&=-\gamma_{j,k}.
\end{align*}
Define the weight $\wt(T)$ of $T$ to be the sum of the weights of all the unshaded boxes of $T$.
\end{Definition}

Note that the unique element of weight zero, denoted $T_\infty$, is the tableau where all boxes are shaded. For example, in types $B_3$ and $C_3$,

\[
T_\infty = \begin{tikzpicture}[baseline]
\matrix [tab] 
{
	\node[draw,fill=gray!30]{1}; & 
	\node[draw,fill=gray!30]{1}; & 
	\node[draw,fill=gray!30]{1}; \\
	\node[draw,fill=gray!30]{2}; & 
	\node[draw,fill=gray!30]{2}; \\
	\node[draw,fill=gray!30]{3}; \\
};
\end{tikzpicture}.
\]

\begin{Definition}
Let $T\in \TT(\infty)$.
\begin{enumerate}
 \item The \textit{Far-Eastern reading} of $T$, denoted $\rw{FE}(T)$, records the entries of the boxes in the columns of $T$ from top to bottom and proceeding from right to left.
 \item The \textit{Middle-Eastern reading} of $T$, denoted $\rw{ME}(T)$, records the entries of the boxes in the rows of $T$ from right to left and proceeding from top to bottom. 
\end{enumerate}
\label{def:T_MEFEread}
\end{Definition}

\begin{Definition}\label{def:Tisignature}
Let $T\in \TT(\infty)$ of type $B_n$ or $C_n$, and set $\rw{}(T) = \rw{ME}(T)$ or $\rw{FE}(T)$.  Consider the fundamental crystals from \eqref{fundBC}.  For each $i\in I=\{1,2, \dots ,n \}$, the bracketing sequence $\BR_i(T)$ is obtained by replacing each letter in $\rw{}(T)$ with $)^p(^q$, where $p$ is number of consecutive $i$-arrows entering and $q$ is the number of consecutive $i$-arrows leaving the corresponding box in the fundamental crystal. 

After determining $\BR_i(T)$, sequentially cancel all ()-pairs to obtain a sequence of the form $)\cdots)(\cdots($ called the \textit{$i$-signature} of $T$. The $i$-signature is denoted as $\BR_i^c(T)$.
\end{Definition}

\begin{Definition}\label{def:T-ops}
Let $T\in \TT(\infty)$ and $i\in I$.  Define $\boldsymbol{0}$ as a formal object not in $\TT(\infty)$.
\begin{enumerate}
 \item If there is no `$)$' in $\BR_i^c(T)$ then set $e_iT=\boldsymbol{0}$.  Otherwise let $\mybox{r}$ be the box in $T$ corresponding to the rightmost `$)$' in $\BR_i^c(T)$.  Define $e_iT$ to be the tableau obtained from $T$ by replacing the $r$ in $\mybox{r}$ with the predecessor in the alphabet of $\TT(\infty)$. If this creates a column with exactly the entries $1,2,\ldots ,i$, then delete that column. 
 \item Let $\mybox{\ell}$ be the box in $T$ corresponding to leftmost `$($' in $\BR_i^c(T)$.   Define $f_iT$ to be the tableau obtained from $T$ by replacing the $\ell$ in $\mybox{\ell}$ with the successor of $\ell$ in the alphabet of $\TT(\infty)$.  If $\mybox{\ell}$ occurs in row $i$ and $\ell=i$, then also insert a column with the entries $1,2,\dots,i$ directly to the left of $\mybox{\ell}$. 
 \end{enumerate}
\end{Definition}

\begin{Example} \label{ex:BefAction}
	Let $T\in \TT(\infty)$ for $\g$ of type $B_3$ where
	\[
	T = \begin{tikzpicture}[baseline]
	\matrix [tab] 
	{
		\node[draw,fill=gray!30]{1}; & 
		\node[draw,fill=gray!30]{1}; & 
		\node[draw,fill=gray!30]{1}; & 
		\node[draw,fill=gray!30]{1}; & 
		\node[draw,fill=gray!30]{1}; & 
		\node[draw,fill=gray!30]{1}; & 
		\node[draw,fill=gray!30]{1}; & 
		\node[draw,fill=gray!30]{1}; & 
		\node[draw,fill=gray!30]{1}; & 
		\node[draw]{2}; & 
		\node[draw]{0}; &
		\node[draw]{\overline 3}; &
		\node[draw]{\overline 2}; & 
		\node[draw]{\overline 1}; & 
		\node[draw]{\overline 1};\\
		\node[draw,fill=gray!30]{2}; & 
		\node[draw,fill=gray!30]{2}; & 
		\node[draw,fill=gray!30]{2}; & 
		\node[draw,fill=gray!30]{2}; & 
		\node[draw]{3}; & 
		\node[draw]{0}; &
		\node[draw]{\overline 2}; &
		\node[draw]{\overline 2}; \\
		\node[draw,fill=gray!30]{3}; & 
		\node[draw]{\overline 3}; &  
		\node[draw]{\overline 3}; \\
	};
	\end{tikzpicture}\ .
	\]
By Definition \ref{def:Tisignature}, we have 
	\[
	\arraycolsep=2pt
	\begin{array}{rcccccccccccccccccccccccccccl}
	\rw{ME}(T)=& \overline1 & \overline1 & \overline2 & \overline3 & 0 & 2 & 1 & 1& 1& 1& 1& 1& 1& 1& 1&
	\overline2 &\overline2 & 0 & 3 & 2 & 2 & 2 & 2 &
	\overline3 &\overline3 & 3 & \\[2pt]
	\BR_3(T) = &  &  &  & )) & )\color{red}{(} \tikzmark{left1}{} & &  & & & & & & & & &
	& &\,^{\tikzmark{right1}{}}\color{red}{)} \color{red}{(} \tikzmark{left2}{} & \color{red}{((} &  &  &  &  &
	\color{red}{))} & \,^{\tikzmark{right2}{}}\color{red}{)}\color{blue}{)} & {\color{blue} (}( &  \\
	\BR_3^c(T) = & & & &  )) & ) & & & & & &  & & & & & & & & & & & & & & \color{blue}{)} & {\color{blue}(} ( & & ,
	\end{array}\\
	\]
	so by Definition \ref{def:T-ops}, we obtain 
\[
e_3T =
\begin{tikzpicture}[baseline]
		\matrix [tab] 
		{
			\node[draw,fill=gray!30]{1}; & 
			\node[draw,fill=gray!30]{1}; & 
			\node[draw,fill=gray!30]{1}; & 
			\node[draw,fill=gray!30]{1}; & 
			\node[draw,fill=gray!30]{1}; & 
			\node[draw,fill=gray!30]{1}; & 
			\node[draw,fill=gray!30]{1}; & 
			\node[draw,fill=gray!30]{1}; & 
			\node[draw,fill=gray!30]{1}; & 
			\node[draw]{2}; & 
			\node[draw]{0}; &
			\node[draw]{\overline 3}; &
			\node[draw]{\overline 2}; & 
			\node[draw]{\overline 1}; & 
			\node[draw]{\overline 1};\\
			\node[draw,fill=gray!30]{2}; & 
			\node[draw,fill=gray!30]{2}; & 
			\node[draw,fill=gray!30]{2}; & 
			\node[draw,fill=gray!30]{2}; & 
			\node[draw]{3}; & 
			\node[draw]{0}; &
			\node[draw]{\overline 2}; &
			\node[draw]{\overline 2}; \\
			\node[draw,fill=gray!30]{3}; & 
			\node[draw]{0}; &  
			\node[draw]{\overline 3}; \\
		};
		\end{tikzpicture}\ 
\]
and
\[
	f_3T =
	\begin{tikzpicture}[baseline]
	\matrix [tab] 
	{
		\node[draw,fill=gray!30]{1}; & 
		\node[draw,fill=gray!30]{1}; & 
		\node[draw,fill=gray!30]{1}; & 
		\node[draw,fill=gray!30]{1}; & 
		\node[draw,fill=gray!30]{1}; & 
		\node[draw,fill=gray!30]{1}; & 
		\node[draw,fill=gray!30]{1}; & 
		\node[draw,fill=gray!30]{1}; & 
		\node[draw,fill=gray!30]{1}; & 
		\node[draw,fill=gray!30]{1}; &
		\node[draw]{2}; & 
		\node[draw]{0}; &
		\node[draw]{\overline 3}; &
		\node[draw]{\overline 2}; & 
		\node[draw]{\overline 1}; & 
		\node[draw]{\overline 1};\\
		\node[draw,fill=gray!30]{2}; &
		\node[draw,fill=gray!30]{2}; & 
		\node[draw,fill=gray!30]{2}; & 
		\node[draw,fill=gray!30]{2}; & 
		\node[draw,fill=gray!30]{2}; & 
		\node[draw]{3}; & 
		\node[draw]{0}; &
		\node[draw]{\overline 2}; &
		\node[draw]{\overline 2}; \\
		\node[draw,fill=gray!30]{3}; & 
		\node[draw]{0}; &  
		\node[draw]{\overline 3};&
		\node[draw]{\overline 3}; \\
	};
\end{tikzpicture}\ .
\]
\end{Example}

\begin{Example} \label{ex:CefAction}
	Let $T\in \TT(\infty)$ for $\g$ of type $C_3$ where
	\[
	T = 	\begin{tikzpicture}[baseline]
	\matrix [tab] 
	{
		\node[draw,fill=gray!30]{1}; & 
		\node[draw,fill=gray!30]{1}; & 
		\node[draw,fill=gray!30]{1}; & 
		\node[draw,fill=gray!30]{1}; & 
		\node[draw,fill=gray!30]{1}; & 
		\node[draw,fill=gray!30]{1}; & 
		\node[draw,fill=gray!30]{1}; & 
		\node[draw,fill=gray!30]{1}; & 
		\node[draw,fill=gray!30]{1}; & 
		\node[draw]{2}; & 
		\node[draw]{3}; &
		\node[draw]{3}; &
		\node[draw]{\overline 3}; &
		\node[draw]{\overline 2}; &
		\node[draw]{\overline 1}; \\
		\node[draw,fill=gray!30]{2}; &
		\node[draw,fill=gray!30]{2}; &
		\node[draw,fill=gray!30]{2}; &
		\node[draw,fill=gray!30]{2}; &
		\node[draw]{3};&
		\node[draw]{\overline 3}; &
		\node[draw]{\overline 3}; &
		\node[draw]{\overline 1}; \\
		\node[draw,fill=gray!30]{3};&
		\node[draw]{\overline 3}; &
		\node[draw]{\overline 3};\\
	};
	\end{tikzpicture}\ .
	\]
By Definition \ref{def:Tisignature}, we have 
	\[
		\arraycolsep=2pt
	\begin{array}{rcccccccccccccccccccccccccccl}
	\rw{ME}(T) =& \overline{1} & \overline2 & \overline3 & 3 & 3 & 1 & 1 & 1 & 1 & 1 & 1 & 1 & 1 & 1 &
	\overline1 & \overline3 & \overline3 & 3 & 2 & 2 & 2 & 2 &
	\overline3 &\overline3 & 3 &  \\
	\BR_3(T) =& & & ) &\color{red}{(}& \color{red}{(} & & & & & & & & & & & \color{red}{)}& \color{red}{)}& \color{red}{(} & & & & & \color{red}{)} & \color{blue}{)} & \color{blue}{(} &  \\
	\BR_3^c(T) =& & & ) & & & & & & & & & & & & & & & & & & & & & \color{blue}{)} & \color{blue}{(} & ,
	\end{array}\\
	\]
so by Definition \ref{def:T-ops}, we obtain
		\[
		e_3T =	\begin{tikzpicture}[baseline]
		\matrix [tab] 
		{
			\node[draw,fill=gray!30]{1}; & 
			\node[draw,fill=gray!30]{1}; & 
			\node[draw,fill=gray!30]{1}; & 
			\node[draw,fill=gray!30]{1}; & 
			\node[draw,fill=gray!30]{1}; & 
			\node[draw,fill=gray!30]{1}; & 
			\node[draw,fill=gray!30]{1}; & 
			\node[draw,fill=gray!30]{1}; & 
			\node[draw]{2}; & 
			\node[draw]{3}; &
			\node[draw]{3}; &
			\node[draw]{\overline 3}; &
			\node[draw]{\overline 2}; &
			\node[draw]{\overline 1}; \\
			\node[draw,fill=gray!30]{2}; &
			\node[draw,fill=gray!30]{2}; &
			\node[draw,fill=gray!30]{2}; &
			\node[draw]{3};&
			\node[draw]{\overline 3}; &
			\node[draw]{\overline 3}; &
			\node[draw]{\overline 1}; \\
			\node[draw,fill=gray!30]{3};&
			\node[draw]{\overline 3};\\
		};
		\end{tikzpicture}\ 	
		\]
and
		\[
		f_3T =\begin{tikzpicture}[baseline]
		\matrix [tab] 
		{
			\node[draw,fill=gray!30]{1}; & 
			\node[draw,fill=gray!30]{1}; & 
			\node[draw,fill=gray!30]{1}; & 
			\node[draw,fill=gray!30]{1}; & 
			\node[draw,fill=gray!30]{1}; & 
			\node[draw,fill=gray!30]{1}; & 
			\node[draw,fill=gray!30]{1}; & 
			\node[draw,fill=gray!30]{1}; & 
			\node[draw,fill=gray!30]{1}; & 
			\node[draw,fill=gray!30]{1}; &
			\node[draw]{2}; & 
			\node[draw]{3}; &
			\node[draw]{3}; &
			\node[draw]{\overline 3}; &
			\node[draw]{\overline 2}; &
			\node[draw]{\overline 1}; \\
			\node[draw,fill=gray!30]{2}; &
			\node[draw,fill=gray!30]{2}; &
			\node[draw,fill=gray!30]{2}; &
			\node[draw,fill=gray!30]{2}; &
			\node[draw,fill=gray!30]{2}; &
			\node[draw]{3};&
			\node[draw]{\overline 3}; &
			\node[draw]{\overline 3}; &
			\node[draw]{\overline 1}; \\
			\node[draw,fill=gray!30]{3};&
			\node[draw]{\overline 3};&
			\node[draw]{\overline 3}; &
			\node[draw]{\overline 3};\\
		};
		\end{tikzpicture}\ .
		\]
\end{Example}

\begin{Theorem}[\cite{HL08}]\label{prop:T-is-crystalFE}
	Using $\rw{FE}(T)$ and the operations defined in Definition \ref{def:T-ops}, $\TT(\infty)$ is a crystal isomorphic to $B(\infty)$.
	\qed
\end{Theorem}

It turns out that using $\rw{ME}$ in place of  $\rw{FE}$ is more convenient for us, and we can do this because of the following: 
\begin{Proposition} \label{prop:row-reading}
Let $\TT(\infty)$ be the set of marginally large tableaux of type $B_n$ or $C_n$.  Then the crystal structures on $\TT(\infty)$ using either $\rw{FE}$ or $\rw{ME}$ are identical.
\end{Proposition}
\begin{proof} Fix $T \in \TT(\infty)$ and $i\in I$. 
By the definition of $e_i$ and $f_i$, we must show that the leftmost `(' and the rightmost `)' in $\BR_i^c(T)$  correspond to the same box for the two different readings.
We need only consider the positions of the $\mybox{i}$, $\mylongbox{i+1}$, $\mybox{\overline{i}}$, $\mylongbox{\overline{i+1}}$, and $\mybox{0}$ (if $i=n$ and $T$ is of type $B_n$). By Definition \ref{defMLT}, these all occur in the first $i+1$ rows. 

The unshaded boxes are read in the same order under the two readings (since there cannot be two in the same column, and if one box is to the left of another it is also weakly below it). Thus the two bracketing sequences are identical until the first shaded $\mybox{i}$ is read. We will call that part of the sequences the prefix. After that, the sequences are as follows, where we use $\ell_{i,j}$ to denote the number of $\mybox{j}$ in row $i$ (and if $i=n$, $\ell_{i+1, ?}$ is taken to be 0):
\begin{align*}
& \text{Middle-Eastern: }
\cdots
(^{\ell_{i,i}}\,
(^{\ell_{i+1,\overline{i+1}}}\,
)^{\ell_{i+1,i+1}},\\
& \text{Far-Eastern: } \cdots
(^{\ell_{i,i}+\ell_{i+1,\overline{i+1}}-\ell_{i+1,i+1}}\,
\underbrace{()\cdots()}_{\ell_{i+1,i+1}}. 	
\end{align*}
Since $\ell_{i,i}>\ell_{i+1,i+1}$, after cancellation these parts of the sequences contain only `$($' and the leftmost `$($' corresponds to the rightmost $\mybox{i}$ in row $i$.  

Thus if the prefix has an uncanceled `(', then this remains uncanceled in both complete bracketing sequences, and corresponds to the same box for both. If the prefix does not have an uncanceled `(', then in both readings the leftmost uncanceled `(' comes from the rightmost $\mybox{i}$ in row $i$. Furthermore, the sequences only have an uncanceled `)' if this comes from the prefix, in which case it corresponds to the same box in both. 
\end{proof}

\subsection{Crystal of Kostant partitions}

Here we review the crystal structure on Kostant partitions from \cite{SST1}. As explained there, this is naturally identified with the crystal of PBW monomials as given in \cite{BZ01,L10} (see also \cite{Tin}) for the reduced expression
\[
w_0 = 
(s_1 s_2 \cdots s_{n-2} s_{n-1} s_n s_{n-2} \cdots s_1)\cdots 
(s_{n-2} s_{n-1} s_n s_{n-2}) s_{n-1} s_n.
\]

Let $\R$ be the set of symbols $\{ (\beta) : \beta\in\Phi^+\}$.  Let $\Kp(\infty)$ be the free $\bz_{\ge0}$-span of $\R$.  This is the set of {\it Kostant partitions}. Elements of $\Kp(\infty)$ are written in the form $\bm\alpha = \sum_{(\beta)\in \R} c_\beta(\beta)$.

\begin{Definition}\label{def:KPisignature}
Consider the following sequences of positive roots depending on $i\in I$ for type $B_n$ or $C_n$.  For $1 \le i \le n-1$, define
\begin{align*}
\Phi_i^B=\Phi_i^C  &= ( \beta_{1,i} ,\beta_{1,i-1}  , \gamma_{1,i}, \gamma_{1,i+1} , \dots ,  \beta_{i-1,i} , \beta_{i-1,i-1} , \gamma_{i-1,i} ,\gamma_{i-1,i+1} , \beta_{i,i} ), \\
\Phi_{n}^B &= (\beta_{1,n},\beta_{1,n-1},\gamma_{1,n} , \beta_{1,n} , \dots ,\beta_{n-1,n} , \beta_{n-1,n-1} ,  \gamma_{n-1,n} , \beta_{n-1,n} ,\beta_{n,n} ), \\
\Phi_{n}^C &= (\gamma_{1,1}, \beta_{1,n-1},\gamma_{1,n} ,\gamma_{1,1} , \dots ,  \gamma_{n-1,n-1} ,\beta_{n-1,n-1} , \gamma_{n-1,n} , \gamma_{n-1,n-1} ,\gamma_{n,n} ). 
\end{align*}
Let $\bm\alpha \in \Kp(\infty)$.  Define the bracketing sequence $S_i(\bm{\alpha})$ by replacing the roots in $\Phi_i^B$ or $\Phi_i^C$ with left and right brackets 
as follows:

In type $B_n$ and $C_n$ with $1\leq i < n$, set
\[ 
\begin{array}{rcccccccccc}
S_i(\bm\alpha) = &
\underbrace{)\cdots)}_{c_{\beta_{1,i}}}\!\! & \!\! 
\underbrace{(\cdots(}_{c_{\beta_{1,i-1}}} \!\! & \!\!   
\underbrace{)\cdots)}_{c_{\gamma_{1,i}}} \!\! & \!\!  
\underbrace{(\cdots(}_{c_{\gamma_{1,i+1}}} &  
\cdots &
\underbrace{)\cdots)}_{c_{\beta_{i-1,i}}} \!\! & \!\!  
\underbrace{(\cdots(}_{c_{\beta_{i-1,i-1}}} \!\! & \!\!  
\underbrace{)\cdots)}_{c_{\gamma_{i-1,i}}} \!\! & \!\!  
\underbrace{(\cdots(}_{c_{\gamma_{i-1,i+1}}} \!\! & \!\!  
\underbrace{)\cdots)}_{c_{\beta_{i,i}}}
\end{array}.
\]
In type $B_n$ with $i=n$, set
\[
\begin{array}{rcccccccccc}
\arraycolsep=0pt
S_n(\bm\alpha) = &
\underbrace{)\cdots)}_{c_{\beta_{1,n}}} \!\! & \!\!  
\underbrace{(\cdots(}_{2c_{\beta_{1,n-1}}} \!\! & \!\!   
\underbrace{)\cdots)}_{2c_{\gamma_{1,n}}} \!\! & \!\!
\underbrace{(\cdots(}_{c_{\beta_{1,n}}} &  
\cdots &
\underbrace{)\cdots)}_{c_{\beta_{n-1,n}}} \!\! & \!\!
\underbrace{(\cdots(}_{2c_{\beta_{n-1,n-1}}} \!\! & \!\!
\underbrace{)\cdots)}_{2c_{\gamma_{n-1,n}}} \!\! & \!\!
\underbrace{(\cdots(}_{c_{\beta_{n-1,n}}} \!\! & \!\!
\underbrace{)\cdots)}_{c_{\beta_{n,n}}} 
\end{array}.
\]
In type $C_n$ with $i=n$, set 
\[
\begin{array}{rcccccccccc}
\arraycolsep=0pt
S_n(\bm\alpha) = &
\underbrace{)\cdots)}_{c_{\gamma_{1,1}}} \!\! & \!\!
\underbrace{(\cdots(}_{c_{\beta_{1,n-1}}} \!\! & \!\! 
\underbrace{)\cdots)}_{c_{\gamma_{1,n}}} \!\! & \!\!
\underbrace{(\cdots(}_{c_{\gamma_{1,1}}}  &
\cdots &
\underbrace{)\cdots)}_{c_{\gamma_{n-1,n-1}}} \!\! & \!\!
\underbrace{(\cdots(}_{c_{\beta_{n-1,n-1}}} \!\! & \!\!
\underbrace{)\cdots)}_{c_{\gamma_{n-1,n}}} \!\! & \!\!
\underbrace{(\cdots(}_{c_{\gamma_{n-1,n-1}}} \!\! & \!\!
\underbrace{)\cdots)}_{c_{\gamma_{n,n}}} \ .
\end{array}
\]
In each case successively cancel all $()$-pairs in $S_i(\bm\alpha)$ to obtain a sequence of the form $)\cdots)(\cdots($ which we call the $i$-signature of $\bm{\alpha}$ denoted by $S_i^c(\bm\alpha)$.
\end{Definition}

\begin{Remark}
Roughly, left brackets correspond to roots $\beta \in \Phi_i$ such that $\beta+\alpha_i$ is a root and right brackets correspond to roots $\beta\in\Phi_i$ such that $\beta-\alpha_i$ is a root (or $\beta = \alpha_i$) except when $i=n$, where some subtleties arise. 
\end{Remark}

\begin{Definition} \label{def:KPops}
Let $i\in I$ and $\bm\alpha\in\Kp(\infty)$ with $\bm\alpha = \sum_{(\beta)\in \R} c_\beta(\beta)\in\Kp(\infty)$. 
\begin{itemize}
\item Define $\wt(\bm\alpha) = -\sum_{\beta\in\Phi^+} c_\beta\beta$.
\item Define $\varepsilon_i(\bm\alpha) = $ number of uncanceled `)' in $S_i(\bm\alpha)$.
\item Define $\varphi_i(\bm\alpha) = \varepsilon_i(\bm\alpha) + \langle \alpha_i^\vee , \wt(\bm\alpha) \rangle$.
\end{itemize}
The following two rules hold except in  the case where $\g$ is of type $C_n$ and $i= n$.
\begin{itemize}
\item Let $\beta$ be the root corresponding to the rightmost `$)$' in $S_i^c(\bm\alpha)$.  Define
\[
e_i\bm\alpha = \bm\alpha - (\beta) + (\beta-\alpha_i).
\]
Note that if $\beta=\alpha_i$, we interpret $(0)$ as the additive identity in $\Kp(\infty)$.  Furthermore, if no such `$)$' exists, then $e_i\bm\alpha = \boldsymbol{0}$, where $\boldsymbol{0}$ is a formal object not contained in $\Kp(\infty)$. 
\item Let $\gamma$ denote the root corresponding to the leftmost `$($' in $S_i^c(\bm\alpha)$.  Define,
\[
f_i\bm\alpha = \bm\alpha - (\gamma) + (\gamma+\alpha_i).
\]
If no such `$($' exists, set $f_i\bm\alpha = \bm\alpha + (\alpha_i)$.
\end{itemize}
If $\g$ is of type $C_n$, then $e_n$ and $f_n$ are defined as follows.
\begin{itemize}
\item Let $\beta$ be the root corresponding to the rightmost `$)$' in $S_n^c(\bm\alpha)$.  Define $e_n\bm\alpha$ as follows, for $k \in \{ 1, \dots, n-1 \}$.  If no such $\beta$ exists, then $e_n\bm\alpha = \boldsymbol{0}$.   
\begin{enumerate}
\item If $\beta = \gamma_{k,n}$ and $c_{\gamma_{k,n}}=c_{\beta_{k,n-1}}+1$, then
$
e_n\bm\alpha = \bm\alpha - (\beta) + (\beta_{k,n-1}).
$
\item If $\beta = \gamma_{k,n}$ and $c_{\gamma_{k,n}}>c_{\beta_{k,n-1}}+1$, then
$
e_n\bm\alpha = \bm\alpha - 2(\beta) +  (\gamma_{k,k}).
$
\item If $\beta = \gamma_{k,k}$, then
$
e_n\bm\alpha = \bm\alpha -(\beta) + 2(\beta_{k,n-1}).
$
\item If $\beta = \gamma_{n,n}$, then 
$
e_n \bm\alpha = \bm\alpha - (\beta).
$
\end{enumerate}
\item Let $\gamma$ denote the root corresponding to the leftmost `$($' in $S_n^c(\bm\alpha)$.  Define $f_n\bm\alpha$ as follows, for $k \in \{ 1, \dots, n \}$.  If no such $\gamma$ exists, then $f_n\bm\alpha = \bm\alpha + (\gamma_{n,n})$.
\begin{enumerate}
\item If $\gamma=\beta_{k,n-1}$ and $c_{\gamma_{k,n}}=c_{\beta_{k,n-1}}-1$, then 
$
f_n\bm\alpha = \bm\alpha - (\gamma) + (\gamma_{k,n}).
$
\item If  $\gamma=\beta_{k,n-1}$ and $c_{\gamma_{k,n}}<c_{\beta_{k,n-1}}-1$, then 
$
f_n\bm\alpha = \bm\alpha - 2(\gamma) + (\gamma_{k,k}).
$
\item If $\gamma=\gamma_{k,k}$, then
$
f_n\bm\alpha = \bm\alpha - (\gamma) + 2(\gamma_{k,n}).
$
\end{enumerate}
\end{itemize}
\end{Definition}

\begin{Example}\label{ex:KPopsC_2f3e}
Let $\Kp(\infty)$ be of type $C_3$ and let $\bm\alpha\in \Kp(\infty)$, where
\[
\bm\alpha = 4(\beta_{1,2})+2(\gamma_{1,3})+2(\gamma_{1,1})+(\gamma_{2,2})+(\gamma_{2,3})+(\gamma_{3,3}).
\]
We consider the action of $f_3$, so we must first compute the bracketing sequence: 
\[
\begin{array}{ccccccccccl}
& c_{\gamma_{1,1}} & c_{\beta_{1,2}} & c_{\gamma_{1,3}} & c_{\gamma_{1,1}} & c_{\gamma_{2,2}} & c_{\beta_{2,2}} & c_{\gamma_{2,3}} & c_{\gamma_{2,2}} & c_{\gamma_{3,3}}\\
S_3(\bm\alpha)= & )) & (( \color{red}{((} & \color{red}{))} & \color{red}{((} & \color{red}{)} & & \color{red}{)} & \color{red}{(} & \color{red}{)} \\ 
S_3^c(\bm\alpha)=& ))& {\color{blue}(} ( &&&&&&&&.
\end{array}
\] 
Hence $f_3\bm\alpha = 2(\beta_{1,2})+2(\gamma_{1,3})+3(\gamma_{1,1})+(\gamma_{2,2})+(\gamma_{2,3})+(\gamma_{3,3})$.
\end{Example}

\begin{Example}\label{ex:KPopsC_3f2e}
	Let $\Kp(\infty)$ be of type $C_3$ and let $\bm\alpha\in \Kp(\infty)$, where
	\[
	\bm\alpha = 2(\beta_{1,2})+2(\gamma_{1,3})+3(\gamma_{1,1})+(\gamma_{2,2})+(\gamma_{2,3})+(\gamma_{3,3}).
	\]
	To compute $f_3\bm\alpha$ we first need the relevant bracketing sequence, which is  
	\[
	\begin{array}{ccccccccccl}
	& c_{\gamma_{1,1}} & c_{\beta_{1,2}} & c_{\gamma_{1,3}} & c_{\gamma_{1,1}} & c_{\gamma_{2,2}} & c_{\beta_{2,2}} & c_{\gamma_{2,3}} & c_{\gamma_{2,2}} & c_{\gamma_{3,3}}\\
	S_3(\bm\alpha)= & ))) &  \color{red}{((} & \color{red}{))} & ({\color{red}((} & \color{red}{)} & & \color{red}{)} & \color{red}{(} & \color{red}{)} \\ 
	S_3^c(\bm\alpha)=& )))& &  & { \color{blue}(} &&&&&& .
	\end{array}
	\] 
Hence $f_3\bm\alpha = 2(\beta_{1,2})+4(\gamma_{1,3})+2(\gamma_{1,1})+(\gamma_{2,2})+(\gamma_{2,3})+(\gamma_{3,3})$.
\end{Example}

\begin{Proposition}[\cite{SST1}]\label{prop:kp-is-crystal}
Using the operators defined in Definition \ref{def:KPops}, the set $\Kp(\infty)$ is a crystal isomorphic to $B(\infty)$. \qed
\end{Proposition}

\section{The isomorphism}\label{sec:MLTtoKPtypeBC}

Our isomorphism $\Psi$ is given as a reversible algorithm to construct an element of $\Kp(\infty)$ from an element of $\TT(\infty)$.  We prove that $\Psi$ preserves the crystal structure by using the fact that under the Middle-Eastern reading the bracketing sequence for marginally large tableaux factors by row. Throughout we restrict to types $B_n$ and $C_n$.

\begin{Theorem} \label{isom_BC}
 Define $\Psi\colon \TT(\infty) \longrightarrow \Kp(\infty)$ by the following process. Fix $T\in \TT(\infty)$ and let $R_1,\dots,R_{n}$ denote the rows of $T$ starting at the top.  Set $\Psi(T) = \sum_{j=1}^{n} \Psi(R_j)$, where $\Psi(R_j)$ is defined as follows. 
If $T$ is of type $B_n$:
\begin{enumerate}
	\item each pair $\Big(\mybox{n},\mybox{\overline{n}}\Big)$ maps to $2(\beta_{j,n})$;
	\item each $\mybox{0}$ maps to $(\beta_{j,n})$;
	\item if $j=n$, then each $\mybox{\overline{n}}$ maps to $2(\beta_{n,n})$.
\end{enumerate}
If $T$ is of type $C_n$:
\begin{enumerate}
	\setcounter{enumi}{3}
	\item each pair $\Big(\mybox{n},\mybox{\overline{n}}\Big)$ maps to $(\gamma_{j,j})$;
	\item if $j=n$, then each $\mybox{\overline{j}}$ maps to $(\gamma_{n,n})$.
\end{enumerate}
For all remaining boxes:
\begin{enumerate}
\setcounter{enumi}{5}
\item $\mybox{\overline{j}}$ maps to $(\beta_{j,j})+(\gamma_{j,j+1})$;
\item each pair $\left(\mybox{k},\mybox{\overline{k}}\right)$, where $j<k<n$, maps to $(\beta_{j,k})+(\gamma_{j,k+1})$;
\item each unpaired $\mybox{k}$ maps to $(\beta_{j,k-1})$, for $k \in \{j+1,\dots,n\}$;
\item each unpaired $\mybox{\overline{k}}$ maps to $(\gamma_{j,k})$, for $\overline{k} \in \{\overline{n},\dots,\overline{j+1}\}$.
\end{enumerate}
Then $\Psi$ is a crystal isomorphism.
\end{Theorem}

The proof of Theorem \ref{isom_BC} will occupy the rest of this section.  

\begin{Example}\label{ex:typeBisom}
	Let $T$ be the marginally large tableau of type $B_3$ from Example \ref{ex:BefAction}. By Theorem \ref{isom_BC}, 
\[
	\Psi(T) 
	= 2(\beta_{1,1}) + (\beta_{1,2})+ (\beta_{1,3})+2(\gamma_{1,3})+2(\gamma_{1,2})+3(\beta_{2,2})+(\beta_{2,3})+2(\gamma_{2,3})+4(\beta_{1,3}).
\]
Then 
	\[
	\arraycolsep=4pt
	\begin{array}{cccccccccl}
	& c_{\beta_{1,3}} & 2c_{\beta_{1,2}} & 2c_{\gamma_{1,3}} & c_{\beta_{1,3}} & c_{\beta_{2,3}} & 2c_{\beta_{2,2}} & 2c_{\gamma_{2,3}} & c_{\beta_{3,3}} \\[2pt]
	S_3(\Psi(T))=& ) & {\color{red}((}  & {\color{red}))})) & {\color{red}(} & {\color{red}(}  & {\color{red} ((}   {\color{red}((((} & {\color{red}))))} & {\color{red}) )}  {\color{red} ))}\\
	S_3^c(\Psi(T))=& ) &   & )\color{blue}{)} &  &   &    &  & & ,
	\end{array}
	\] 
	so $f_3\Psi(T)=\Psi(T)+(\beta_{3,3})$, which agrees with 
	\[
	  	\Psi(f_3 T)
	= 2(\beta_{1,1}) + (\beta_{1,2})+ (\beta_{1,3})+2(\gamma_{1,3})+2(\gamma_{1,2})+3(\beta_{2,2})+(\beta_{2,3})+2(\gamma_{2,3})+5(\beta_{1,3}).
\]
\end{Example}

\begin{Example}
Consider type $C_3$ and
	\[
	T = \begin{tikzpicture}[baseline]
	\matrix [tab] 
	{
		\node[draw,fill=gray!30]{1}; & 
		\node[draw,fill=gray!30]{1}; & 
		\node[draw,fill=gray!30]{1}; & 
		\node[draw,fill=gray!30]{1}; & 
		\node[draw,fill=gray!30]{1}; & 
		\node[draw,fill=gray!30]{1}; & 
		\node[draw,fill=gray!30]{1}; & 
		\node[draw]{2}; & 
		\node[draw]{2}; & 
		\node[draw]{3}; &
		\node[draw]{3}; & 
		\node[draw]{3}; &
		\node[draw]{3}; &
		\node[draw]{\overline 3}; & 
		\node[draw]{\overline 3}; & 
		\node[draw]{\overline 2};& 
		\node[draw]{\overline 2};\\
		\node[draw,fill=gray!30]{2}; &
		\node[draw,fill=gray!30]{2}; & 
		\node[draw,fill=gray!30]{2}; & 
		\node[draw]{3}; & 
		\node[draw]{\overline 3}; & 
		\node[draw]{\overline 3}; & \\
		\node[draw,fill=gray!30]{3}; &
		\node[draw]{\overline 3}; \\
	};
	\end{tikzpicture}\ .
	\]
	Then
	\[
	\arraycolsep=2pt
	\begin{array}{rcccccccccccccccccccccccccccc}
	\rw{ME}(T)=& \overline 2 & \overline 2 & \overline 3 & \overline3 & 3 &3 & 3 & 3 & 2 & 2 & 1 &1 &1 &1 &1 &1 &1 
	&\overline 3 & \overline3 & 3 &2 &2 &2 
	&\overline 3 & 3\\
	\BR_3(T)=& & & ) & ) & ( & ( & {\color{red}(} & {\color{red}(} & &  & & & & & & &
	&{\color{red})} & {\color{red})} & {\color{red}(} & & & 
	&{\color{red})} & (\\
	\BR_3^c(T)=& & & ) & ) & {\color{blue}(} & ( & & & & & & & & & & & & & & & & & & & ( & ,
	\end{array}
	\]
	so
	\[
	f_3T =\begin{tikzpicture}[baseline]
	\matrix [tab] 
	{
		\node[draw,fill=gray!30]{1}; & 
		\node[draw,fill=gray!30]{1}; & 
		\node[draw,fill=gray!30]{1}; & 
		\node[draw,fill=gray!30]{1}; & 
		\node[draw,fill=gray!30]{1}; & 
		\node[draw,fill=gray!30]{1}; & 
		\node[draw,fill=gray!30]{1}; & 
		\node[draw]{2}; & 
		\node[draw]{2}; & 
		\node[draw]{3}; &
		\node[draw]{3}; & 
		\node[draw]{3}; &
		\node[draw]{\overline 3}; &
		\node[draw]{\overline 3}; & 
		\node[draw]{\overline 3}; & 
		\node[draw]{\overline 2};& 
		\node[draw]{\overline 2};\\
		\node[draw,fill=gray!30]{2}; &
		\node[draw,fill=gray!30]{2}; & 
		\node[draw,fill=gray!30]{2}; & 
		\node[draw]{3}; & 
		\node[draw]{\overline 3}; & 
		\node[draw]{\overline 3}; & \\
		\node[draw,fill=gray!30]{3}; &
		\node[draw]{\overline 3}; \\
	};
	\end{tikzpicture}\ .
	\]
	We now apply the isomorphism from Theorem \ref{isom_BC} to $T$ and $f_3 T$ to get 
	\begin{align*}
	\label{ex:typeCisom}
	\Psi(T)	& = 4(\beta_{1,2})+ 2(\gamma_{1,1})+ 2(\gamma_{1,3})+ (\gamma_{2,2})+ (\gamma_{2,3}) + (\gamma_{3,3}), \text{ and }\\
	\Psi (f_3 T) 
	& = 2(\beta_{1,2})+ 3(\gamma_{1,1})+ 2(\gamma_{1,3})+ (\gamma_{2,2})+ (\gamma_{2,3}) + (\gamma_{3,3}).
	\end{align*}
	Note that these are the same Kostant partitions as in Example \ref{ex:KPopsC_2f3e}.   Hence 
	\[
	f_3\Psi(T)=\Psi(T)-2(\beta_{1,2})+(\gamma_{1,1}) =\Psi(f_3 T).
	\]
\end{Example}

Denote by $e_i^\TT$ and $f_i^\TT$ the Kashiwara operators on $\TT(\infty)$ from Definition \ref{def:T-ops}, and $e_i^\Kp$ and $f_i^\Kp$ as the operators on $\Kp(\infty)$ from Definition \ref{def:KPops}.    

\begin{Lemma}\label{lem:onerow_i} 
Fix $i \in I\setminus\{n\}$ and a row index $j$. 
Let $T \in \TT(\infty)$ be such that the only unshaded boxes occur in row $j$. If the leftmost `$($' in $\BR_i^c(T)$ comes from $R_j$, then $f_i^\Kp\Psi(T) = \Psi (f_i^{\TT}T )$.  
\end{Lemma}

\begin{proof}
First consider $i\in\{1,\dots,n-1\}$ and row $R_j$ for $j<i$.  We are only interested in boxes which give rise to brackets in $\BR_i(R_j)$ or $S_i\bigl(\Psi(R_j)\bigr)$.  Following Definition \ref{def:T-ops} these boxes are the pairs $\Bigl(\mylongbox{i-1}, \mylongbox{\overline{\imath-1}}\Bigr)$ and the $\mybox{i}$, $\mylongbox{i+1}$, $\mylongbox{\overline{\imath+1}}$, and $\mybox{\overline\imath}$.
 
A pair $\Bigl(\mylongbox{i-1},\mylongbox{\overline{\imath-1}}\Bigr)$ corresponds to no brackets in $\BR_i(R_j)$, and to $(\beta_{j,i-1})$, $(\gamma_{j,i})$ in $\Psi(R_j)$, corresponding to a canceling pair of brackets in $S_i\bigl(\Psi(R_j)\bigr)$. 
So the statement is true if and only if it is true with these removed. 
Thus we can assume $R_j$ has no such pairs. 

Now, assume row $j$ of $T$ has
$p$, $q$, $r$, and $s$ boxes of $\overline{\imath+1}$, $i+1$, $i$, and $\overline{\imath}$ respectively: 
\[
R_j = 
\begin{tikzpicture}[baseline=-4,xscale=.75,yscale=.65]
\def\pt{.5}

\draw[-] (1-\pt,-\pt) -- (1-\pt,\pt) -- (13+\pt,\pt) -- (13+\pt,-\pt) -- cycle;
\foreach \x in {1.5,2.5,...,12.5}
{\draw[-] (\x,\pt) -- (\x,-\pt);}
\node at (1,0) {$i$};
\node at (2,0) {$\cdots$};
\node at (3,0) {$i$};
\node[font=\scriptsize] at (4,0) {$i+1$};
\node at (5,0) {$\cdots$};
\node[font=\scriptsize] at (6,0) {$i+1$};
\node at (7,0) {$\cdots$};
\node[font=\scriptsize] at (8,0) {$\overline{\imath+1}$};
\node at (9,0) {$\cdots$};
\node[font=\scriptsize] at (10,0) {$\overline{\imath+1}$};
\node at (11,0) {$\overline{\imath}$};
\node at (12,0) {$\cdots$};
\node at (13,0) {$\overline{\imath}$};
\draw [thick,decorate,decoration={brace,amplitude=6pt,mirror,raise=3pt}]
(0.6,-\pt) -- (3.4,-\pt) node[black,midway,yshift=-0.5cm,font=\scriptsize] {$r$};
\draw [thick,decorate,decoration={brace,amplitude=6pt,mirror,raise=3pt}]
(3.6,-\pt) -- (6.4,-\pt) node[black,midway,yshift=-0.5cm,font=\scriptsize] {$q$};
\draw [thick,decorate,decoration={brace,amplitude=6pt,mirror,raise=3pt}]
(7.6,-\pt) -- (10.4,-\pt) node[black,midway,yshift=-0.5cm,font=\scriptsize] {$p$};
\draw [thick,decorate,decoration={brace,amplitude=6pt,mirror,raise=3pt}]
(10.6,-\pt) -- (13.4,-\pt) node[black,midway,yshift=-0.5cm,font=\scriptsize] {$s$};
\end{tikzpicture}
\]
Define $\Psi(R_j)=\sum_{(\beta)\in \R} c_{\beta}(\beta)$.  The general bracketing sequences for both are
\[
\BR_i(R_j)=\  )^s\  (^p\  )^q\  (^r\ , 
\ \ \ \text{ and } \ \ \ 
S_i\bigl(\Psi(R_j)\bigr) =\
)^{c_{\beta_{j,i}}} (^{c_{\beta_{j,i-1}}} )^{c_{\gamma_{j,i}}} (^{c_{\gamma_{j,i+1}}}.
\]

{\bf Case 1: $\boldsymbol{p>q}$, $\boldsymbol{r>s}$ and $\boldsymbol{1\le j<i<n}$.}  
By the definition of $\Psi$,
\begin{align*}
\Psi(R_j) &= q (\beta_{j,i+1}+\gamma_{j,i+2}) + (p-q)(\gamma_{j,i+1}) + (r-s)(\beta_{j,i-1}) + s(\beta_{j,i}+\gamma_{j,i+1})\\
&=(r-s)(\beta_{j,i-1})+s(\beta_{j,i})+q(\beta_{j,i+1})+(s+p-q)(\gamma_{j,i+1})+q(\gamma_{j,i+2}).
\end{align*}
Calculating the action of $f_i^\TT$ on $R_j$ gives
\[
f_i^\TT R_j = 
\begin{tikzpicture}[baseline=-4,xscale=.75,yscale=.65]
\def\pt{.5}
\draw[-] (1-\pt,-\pt) -- (1-\pt,\pt) -- (13+\pt,\pt) -- (13+\pt,-\pt) -- cycle;
\foreach \x in {1.5,2.5,...,12.5}
{\draw[-] (\x,\pt) -- (\x,-\pt);}
\node at (1,0) {$i$};
\node at (2,0) {$\cdots$};
\node at (3,0) {$i$};
\node[font=\scriptsize] at (4,0) {$i+1$};
\node at (5,0) {$\cdots$};
\node[font=\scriptsize] at (6,0) {$i+1$};
\node at (7,0) {$\cdots$};
\node[font=\scriptsize] at (8,0) {$\overline{\imath+1}$};
\node at (9,0) {$\cdots$};
\node[font=\scriptsize] at (10,0) {$\overline{\imath+1}$};
\node at (11,0) {$\overline{\imath}$};
\node at (12,0) {$\cdots$};
\node at (13,0) {$\overline{\imath}$};
\draw [thick,decorate,decoration={brace,amplitude=6pt,mirror,raise=3pt}]
(0.6,-\pt) -- (3.4,-\pt) node[black,midway,yshift=-0.5cm,font=\scriptsize] {$r$};
\draw [thick,decorate,decoration={brace,amplitude=6pt,mirror,raise=3pt}]
(3.6,-\pt) -- (6.4,-\pt) node[black,midway,yshift=-0.5cm,font=\scriptsize] {$q$};
\draw [thick,decorate,decoration={brace,amplitude=6pt,mirror,raise=3pt}]
(7.6,-\pt) -- (10.4,-\pt) node[black,midway,yshift=-0.5cm,font=\scriptsize] {$p-1$};
\draw [thick,decorate,decoration={brace,amplitude=6pt,mirror,raise=3pt}]
(10.6,-\pt) -- (13.4,-\pt) node[black,midway,yshift=-0.5cm,font=\scriptsize] {$s+1$};
\end{tikzpicture}\ .
\]
Then
\begin{align*}
\Psi(f^\TT_iR_j) &= q (\beta_{j,i+1}+\gamma_{j,i+2}) + (p-q-1)(\gamma_{j,i+1}) + (r-s-1)(\beta_{j,i-1}) + (s+1)(\beta_{j,i}+\gamma_{j,i+1})\\
&=(r-s-1)(\beta_{j,i-1})+(s+1)(\beta_{j,i})+q(\beta_{j,i+1})+(s+p-q)(\gamma_{j,i+1})+q(\gamma_{j,i+2}).
\end{align*}
We now apply the operator $f_i^{\Kp}$ to $\Psi(R_j)$ to show equivalence.  In $S_i^c\bigl(\Psi(R_j)\bigr)$ the leftmost `(' corresponds to $\beta_{j,i-1}$ so 
\begin{align*}
f_i^{\Kp}\Psi(R_j)&= (r-s-1) (\beta_{j,i-1}) + (s+1)(\beta_{j,i}) + q (\beta_{j,i+1}) + q(\gamma_{j,i+2}) + (s+p-q)(\gamma_{j,i+1}) \\
&= \Psi(f_i^\TT R_j).
\end{align*}

{\bf Case 2: $\boldsymbol{p>q}$, $\boldsymbol{r\leq s}$, and $\boldsymbol{1\le j<i<n}$.}  By the definition of $\Psi$,
\begin{align*}
\Psi(R_j) &= q (\beta_{j,i+1}+\gamma_{j,i+2}) + (p-q)(\gamma_{j,i+1}) + (s-r)(\gamma_{j,i}) + r(\beta_{j,i}+\gamma_{j,i+1})\\
&=r(\beta_{j,i})+q(\beta_{j,i+1})+q(\gamma_{j,i+2})+(r+p-q)(\gamma_{j,i+1})+(s-r)(\gamma_{j,i}).
\end{align*}
By the definition of $f_i^\TT$, we have  
\[
f_i^\TT R_j = 
\begin{tikzpicture}[baseline=-4,xscale=.75,yscale=.65]
\def\pt{.5}
\draw[-] (1-\pt,-\pt) -- (1-\pt,\pt) -- (13+\pt,\pt) -- (13+\pt,-\pt) -- cycle;
\foreach \x in {1.5,2.5,...,12.5}
{\draw[-] (\x,\pt) -- (\x,-\pt);}
\node at (1,0) {$i$};
\node at (2,0) {$\cdots$};
\node at (3,0) {$i$};
\node[font=\scriptsize] at (4,0) {$i+1$};
\node at (5,0) {$\cdots$};
\node[font=\scriptsize] at (6,0) {$i+1$};
\node at (7,0) {$\cdots$};
\node[font=\scriptsize] at (8,0) {$\overline{\imath+1}$};
\node at (9,0) {$\cdots$};
\node[font=\scriptsize] at (10,0) {$\overline{\imath+1}$};
\node at (11,0) {$\overline{\imath}$};
\node at (12,0) {$\cdots$};
\node at (13,0) {$\overline{\imath}$};
\draw [thick,decorate,decoration={brace,amplitude=6pt,mirror,raise=3pt}]
(0.6,-\pt) -- (3.4,-\pt) node[black,midway,yshift=-0.5cm,font=\scriptsize] {$r$};
\draw [thick,decorate,decoration={brace,amplitude=6pt,mirror,raise=3pt}]
(3.6,-\pt) -- (6.4,-\pt) node[black,midway,yshift=-0.5cm,font=\scriptsize] {$q$};
\draw [thick,decorate,decoration={brace,amplitude=6pt,mirror,raise=3pt}]
(7.6,-\pt) -- (10.4,-\pt) node[black,midway,yshift=-0.5cm,font=\scriptsize] {$p-1$};
\draw [thick,decorate,decoration={brace,amplitude=6pt,mirror,raise=3pt}]
(10.6,-\pt) -- (13.4,-\pt) node[black,midway,yshift=-0.5cm,font=\scriptsize] {$s+1$};
\end{tikzpicture}\
.\]
Then
\begin{align*}
\Psi(f^\TT_iR_j) &= q (\beta_{j,i+1}+\gamma_{j,i+2}) + (p-q-1)(\gamma_{j,i+1}) + (s-r+1)(\gamma_{j,i}) + r(\beta_{j,i}+\gamma_{j,i+1})\\
&=r(\beta_{j,i})+q (\beta_{j,i+1})+q(\gamma_{j,i+2})+(r+p-q-1)(\gamma_{j,i+1})+(s-r+1)(\gamma_{j,i}).
\end{align*}
On the other hand, in $S_i^c\bigl(\Psi(R_j)\bigr)$ the leftmost `(' corresponds to $\gamma_{j,i+1}$ so
\begin{align*}
f_i^{\Kp}\Psi(R_j)&=r(\beta_{j,i})+q(\beta_{j,i+1})+q(\gamma_{j,i+2})+(r+p-q-1)(\gamma_{j,i+1}) +(s-r+1)(\gamma_{j,i})\\
&= \Psi(f_i^\TT R_j).
\end{align*}
Furthermore,
\[
\BR_i(R_j) =\  )^s\  (^{p-q}\   (^r\  
\quad
\text{ and }
\quad
S_i\bigl(\Psi(R_j)\bigr) =\  )^r\ )^{s-r}\ (^{r+p-q},
\]
so both $\BR_i^c(R_j) $ and $S_i^c\bigl(\Psi(R_j)\bigr)$ have $s$ `)' and $r+p-q$ `('.

{\bf Case 3: $\boldsymbol{p\leq q}$, $\boldsymbol{r>s}$, and $\boldsymbol{1\le j<i<n}$.}  By the definition of $\Psi$,
\begin{align*}
\Psi(R_j) &= p (\beta_{j,i+1}+\gamma_{j,i+2}) + (q-p)(\beta_{j,i}) + (r-s)(\beta_{j,i-1}) + s(\beta_{j,i}+\gamma_{j,i+1})\\
&=(r-s)(\beta_{j,i-1})+(s+q-p)(\beta_{j,i})+p(\beta_{j,i+1})+p(\gamma_{j,i+2})+s(\gamma_{j,i+1}).
\end{align*}
By the definition of $f_i^\TT$, we have  
\[
f_i^\TT R_j = 
\begin{tikzpicture}[baseline=-4,xscale=.75,yscale=.65]
\def\pt{.5}
\draw[-] (1-\pt,-\pt) -- (1-\pt,\pt) -- (13+\pt,\pt) -- (13+\pt,-\pt) -- cycle;
\foreach \x in {1.5,2.5,...,12.5}
{\draw[-] (\x,\pt) -- (\x,-\pt);}
\node at (1,0) {$i$};
\node at (2,0) {$\cdots$};
\node at (3,0) {$i$};
\node[font=\scriptsize] at (4,0) {$i+1$};
\node at (5,0) {$\cdots$};
\node[font=\scriptsize] at (6,0) {$i+1$};
\node at (7,0) {$\cdots$};
\node[font=\scriptsize] at (8,0) {$\overline{\imath+1}$};
\node at (9,0) {$\cdots$};
\node[font=\scriptsize] at (10,0) {$\overline{\imath+1}$};
\node at (11,0) {$\overline{\imath}$};
\node at (12,0) {$\cdots$};
\node at (13,0) {$\overline{\imath}$};
\draw [thick,decorate,decoration={brace,amplitude=6pt,mirror,raise=3pt}]
(0.6,-\pt) -- (3.4,-\pt) node[black,midway,yshift=-0.5cm,font=\scriptsize] {$r-1$};
\draw [thick,decorate,decoration={brace,amplitude=6pt,mirror,raise=3pt}]
(3.6,-\pt) -- (6.4,-\pt) node[black,midway,yshift=-0.5cm,font=\scriptsize] {$q+1$};
\draw [thick,decorate,decoration={brace,amplitude=6pt,mirror,raise=3pt}]
(7.6,-\pt) -- (10.4,-\pt) node[black,midway,yshift=-0.5cm,font=\scriptsize] {$p$};
\draw [thick,decorate,decoration={brace,amplitude=6pt,mirror,raise=3pt}]
(10.6,-\pt) -- (13.4,-\pt) node[black,midway,yshift=-0.5cm,font=\scriptsize] {$s$};
\end{tikzpicture}\ 
.\]
Then
\begin{align*}
\Psi(f^\TT_iR_j) &= p (\beta_{j,i+1}+\gamma_{j,i+2}) + (q-p+1)(\beta_{j,i}) + (r-s-1)(\beta_{j,i-1}) + s(\beta_{j,i}+\gamma_{j,i+1})\\
&=(r-s-1)(\beta_{j,i-1})+(s+q-p+1)(\beta_{j,i})+p(\beta_{j,i+1})+p(\gamma_{j,i+2})+s(\gamma_{j,i+1}).
\end{align*}
On the other hand, in $S_i^c\bigl(\Psi(R_j)\bigr)$ the leftmost `(' corresponds to $\beta_{j,i-1}$ so
\begin{align*}
f_i^{\Kp}\Psi(R_j)&=(r-s-1)(\beta_{j,i-1})+(s+q-p+1)(\beta_{j,i})+p(\beta_{j,i+1})+p(\gamma_{j,i+2})+s(\gamma_{j,i+1}) \\
&= \Psi(f_i^\TT R_j).
\end{align*}

{\bf Case 4: $\boldsymbol{p\leq q}$, $\boldsymbol{r\leq s}$, and $\boldsymbol{1\le j<i<n}$.}  By the definition of $\Psi$,
\begin{align*}
\Psi(R_j) &= p (\beta_{j,i+1}+\gamma_{j,i+2}) + (q-p)(\beta_{j,i}) + (s-r)(\gamma_{j,i}) + r(\beta_{j,i}+\gamma_{j,i+1})\\
&=(r+q-p)(\beta_{j,i})+p(\beta_{j,i+1})+p(\gamma_{j,i+2})+r(\gamma_{j,i+1})+(s-r)(\gamma_{j,i}).
\end{align*}
If $r=0$, then $f_i$ will act on the rightmost $\mybox{i}$ in $R_i$ of $T$ (see Case 6 for details on this situation).  When $r>0$, by the definition of $f_i^\TT$, we have  
\[
f_i^\TT R_j = 
\begin{tikzpicture}[baseline=-4,xscale=.75,yscale=.65]
\def\pt{.5}
\draw[-] (1-\pt,-\pt) -- (1-\pt,\pt) -- (13+\pt,\pt) -- (13+\pt,-\pt) -- cycle;
\foreach \x in {1.5,2.5,...,12.5}
{\draw[-] (\x,\pt) -- (\x,-\pt);}
\node at (1,0) {$i$};
\node at (2,0) {$\cdots$};
\node at (3,0) {$i$};
\node[font=\scriptsize] at (4,0) {$i+1$};
\node at (5,0) {$\cdots$};
\node[font=\scriptsize] at (6,0) {$i+1$};
\node at (7,0) {$\cdots$};
\node[font=\scriptsize] at (8,0) {$\overline{\imath+1}$};
\node at (9,0) {$\cdots$};
\node[font=\scriptsize] at (10,0) {$\overline{\imath+1}$};
\node at (11,0) {$\overline{\imath}$};
\node at (12,0) {$\cdots$};
\node at (13,0) {$\overline{\imath}$};
\draw [thick,decorate,decoration={brace,amplitude=6pt,mirror,raise=3pt}]
(0.6,-\pt) -- (3.4,-\pt) node[black,midway,yshift=-0.5cm,font=\scriptsize] {$r-1$};
\draw [thick,decorate,decoration={brace,amplitude=6pt,mirror,raise=3pt}]
(3.6,-\pt) -- (6.4,-\pt) node[black,midway,yshift=-0.5cm,font=\scriptsize] {$q+1$};
\draw [thick,decorate,decoration={brace,amplitude=6pt,mirror,raise=3pt}]
(7.6,-\pt) -- (10.4,-\pt) node[black,midway,yshift=-0.5cm,font=\scriptsize] {$p$};
\draw [thick,decorate,decoration={brace,amplitude=6pt,mirror,raise=3pt}]
(10.6,-\pt) -- (13.4,-\pt) node[black,midway,yshift=-0.5cm,font=\scriptsize] {$s$};
\end{tikzpicture}\
.\]
Then
\begin{align*}
\Psi(f^\TT_i(R_j) &= p (\beta_{j,i+1}+\gamma_{j,i+2}) + (q-p+1)(\beta_{j,i}) + (s-r+1)(\gamma{j,i}) + (r-1)(\beta_{j,i}+\gamma_{j,i+1})\\
&=(r+q-p)(\beta_{j,i})+p(\beta_{j,i+1})+p(\gamma_{j,i+2})+(r-1)(\gamma_{j,i+1})+(s-r+1)(\gamma_{j,i}) .
\end{align*}
On the other hand, in $S_i^c\bigl(\Psi(R_j)\bigr)$ the leftmost `(' corresponds to $\gamma_{j,i+1}$ so
\begin{align*}
f_i^{\Kp}\Psi(R_j)&=(r+q-p)(\beta_{j,i})+p(\beta_{j,i+1})+p(\gamma_{j,i+2})+(r-1)(\gamma_{j,i+1})+(s-r+1)(\gamma_{j,i}) \\
&= \Psi(f_i^\TT R_j).
\end{align*}

We now establish the result for $i\in \{1,\dots,n-1\}$ and row $j=i$.  The general bracketing sequences for both are given here:
\begin{align*}
\BR_i(R_i)&=\  )^s\  (^p\  )^q\  (^r\ , & 
S_i\bigl(\Psi(R_i)\bigr)&=\
)^{c_{\beta_{i,i}}}. 
\end{align*}
Following Definition \ref{def:KPops}, since there is no `$($' in $S_i\bigl(\Psi(R_i)\bigr)$ the action of $f_i^{\Kp}$ will always be to add $(\beta_{i,i})$.

{\bf Case 5: $\boldsymbol{p>q}$, and $\boldsymbol{1\le j=i<n}$.} If $p>q$, then the leftmost `$($' comes from an $\mylongbox{\overline{i+1}}$ so $f_i^{\TT}R_i$ sends an  $\mylongbox{\overline{i+1}}$ to a  $\mybox{\overline{i}}$. Since $p>q$ this does not change the number of $\left(\mylongbox{i+1},\mylongbox{\overline{i+1}}\right)$ pairs, so $\Psi(f_i^{\TT}R_i)=\Psi(R_i)+(\beta_{i,i})=f_i^\Kp\Psi(R_i)$.

{\bf Case 6: $\boldsymbol{p\leq q}$, and $\boldsymbol{1\le j=i<n}$.} If $p\leq q$, then the leftmost `$($' comes from a $\mybox{i}$ so $f_i^{\TT}R_i$ sends an $\mybox{i}$ to an $\mylongbox{i+1}$.  Since $p\leq q $ this does not change the number of $\left(\mylongbox{i+1},\mylongbox{\overline{i+1}}\right)$ pairs, so $\Psi(f_i^{\TT}R_i)=\Psi(R_i)+(\beta_{i,i})=f_i^\Kp\Psi(R_i)$.
\end{proof}

\begin{Lemma}\label{lem:onerow_n} 
	Fix a row index $j\in I$. 
	Let $T \in \TT(\infty)$ be such that the only unshaded boxes occur in row $j$. If the leftmost `$($' in $\BR_n^c(T)$ comes from $R_j$, then $f_n^\Kp\Psi(T) = \Psi (f_n^{\TT}T )$.  
\end{Lemma}

\begin{proof} 
Consider $f_n$ and $T$ to be of type $B_n$.  We need only consider the $\mybox{n}$, $\mybox{0}$, and $\mybox{\overline{n}}$ boxes, since a pair $\Bigl(\mylongbox{n-1},\mylongbox{\overline{n-1}}\Bigr)$ corresponds to no brackets in $\BR_n(R_j)$, and to $(\beta_{j,n-1})$, $(\gamma_{j,n})$ in $\Psi(R_j),$ which gives a canceling pair of brackets in $S_n\bigl(\Psi(R_j)\bigr)$.  Assume row $j$ of $T$ has $p$  $\mybox{\overline{n}}$ boxes, $z$ $\mybox{0}$ boxes, and $q$ $\mybox{n}$ boxes:
\[
R_j = 
\begin{tikzpicture}[baseline=-4,xscale=.75,yscale=.65]
\def\pt{.5}
\draw[-] (1-\pt,-\pt) -- (1-\pt,\pt) -- (7+\pt,\pt) -- (7+\pt,-\pt) -- cycle;
\foreach \x in {1.5,2.5,...,6.5}
{\draw[-] (\x,\pt) -- (\x,-\pt);}
\node at (1,0) {$n$};
\node at (2,0) {$\cdots$};
\node at (3,0) {$n$};
\node at (4,0) {$0$};
\node at (5,0) {$\overline{n}$};
\node at (6,0) {$\cdots$};
\node at (7,0) {$\overline{n}$};
\draw [thick,decorate,decoration={brace,amplitude=6pt,mirror,raise=3pt}]
(0.6,-\pt) -- (3.4,-\pt) node[black,midway,yshift=-0.5cm,font=\scriptsize] {$q$};
\draw [thick,decorate,decoration={brace,amplitude=3pt,mirror,raise=3pt}]
(3.6,-\pt) -- (4.4,-\pt) node[black,midway,yshift=-0.5cm,font=\scriptsize] {$z$};
\draw [thick,decorate,decoration={brace,amplitude=6pt,mirror,raise=3pt}]
(4.6,-\pt) -- (7.4,-\pt) node[black,midway,yshift=-0.5cm,font=\scriptsize] {$p$};
\end{tikzpicture}\ .
\]
The bracketing sequences are:
\begin{align*}
\BR_n(R_j)&=\    )^{2p}\  )^z\ (^z\ (^{2q}\  , &
S_n\bigl(\Psi(T)\bigr)&=\
)^{c_{\beta_{j,n}}}  (^{2c_{\beta_{j,n-1}}}  )^{2c_{\gamma_{j,n}}}  (^{c_{\beta_{j,n}}}.
\end{align*}

{\bf Case 1: $\boldsymbol{p\geq q}$, $\boldsymbol{z=0}$, and $\boldsymbol{1\le j< n}$.} 

By the definition of $\Psi$,
\begin{align*}
\Psi(R_j) &= 2q (\beta_{j,n}) + (p-q)(\gamma_{j,n}).
\end{align*}
If $q=0$, then $f_n$ will act on the $\mybox{n}$ in $R_n$ of $T$ (see Case 5 for more details in this situation). If $q>0$ then, by the definition of $f_n^\TT$, 

\[
f_n^\TT R_j = 
\begin{tikzpicture}[baseline=-4,xscale=.75,yscale=.65]
\def\pt{.5}
\draw[-] (1-\pt,-\pt) -- (1-\pt,\pt) -- (7+\pt,\pt) -- (7+\pt,-\pt) -- cycle;
\foreach \x in {1.5,2.5,...,6.5}
{\draw[-] (\x,\pt) -- (\x,-\pt);}
\node at (1,0) {$n$};
\node at (2,0) {$\cdots$};
\node at (3,0) {$n$};
\node at (4,0) {$0$};
\node at (5,0) {$\overline{n}$};
\node at (6,0) {$\cdots$};
\node at (7,0) {$\overline{n}$};
\draw [thick,decorate,decoration={brace,amplitude=6pt,mirror,raise=3pt}]
(0.6,-\pt) -- (3.4,-\pt) node[black,midway,yshift=-0.5cm,font=\scriptsize] {$q-1$};
\draw [thick,decorate,decoration={brace,amplitude=3pt,mirror,raise=3pt}]
(3.6,-\pt) -- (4.4,-\pt) node[black,midway,yshift=-0.5cm,font=\scriptsize] {$1$};
\draw [thick,decorate,decoration={brace,amplitude=6pt,mirror,raise=3pt}]
(4.6,-\pt) -- (7.4,-\pt) node[black,midway,yshift=-0.5cm,font=\scriptsize] {$p$};
\end{tikzpicture}\ .
\]
Since $p\geq q$ there is one less $\left( \mybox{n},\mybox{\overline{n}}\right)$, one more $\mybox{0}$, and one more unpaired $\mybox{\overline{n}}$, so 
\begin{align*}
\Psi(f_n^\TT R_j) &= 2(q-1) (\beta_{j,n}) + (\beta_{j,n})+ (p-q+1)(\gamma_{j,n}),\\
&=(2q-1)(\beta_{j,n})+(p-q+1)(\gamma_{j,n}).
\end{align*}
On the other hand, in $S_n^c\bigl(\Psi(R_j)\bigr)$ the leftmost `(' corresponds to $\beta_{j,n}$ so 
\[
f_n^{\Kp}\Psi(R_j)= (2q-1) (\beta_{j,n}) + (p-q+1)(\gamma_{j,n}) = \Psi(f_n^\TT R_j).
\]

{\bf Case 2:  $\boldsymbol{p<q}$, $\boldsymbol{z=0}$, and $\boldsymbol{1\le j<n}$.}   
By the definition of $\Psi$,
\begin{align*}
\Psi(R_j) &= (q-p)(\beta_{j,n-1})+2p (\beta_{j,n}) .
\end{align*}
By the definition of $f_n^\TT$, we have 
\[
f_n^\TT R_j = 
\begin{tikzpicture}[baseline=-4,xscale=.75,yscale=.65]
\def\pt{.5}
\draw[-] (1-\pt,-\pt) -- (1-\pt,\pt) -- (7+\pt,\pt) -- (7+\pt,-\pt) -- cycle;
\foreach \x in {1.5,2.5,...,6.5}
{\draw[-] (\x,\pt) -- (\x,-\pt);}
\node at (1,0) {$n$};
\node at (2,0) {$\cdots$};
\node at (3,0) {$n$};
\node at (4,0) {$0$};
\node at (5,0) {$\overline{n}$};
\node at (6,0) {$\cdots$};
\node at (7,0) {$\overline{n}$};
\draw [thick,decorate,decoration={brace,amplitude=6pt,mirror,raise=3pt}]
(0.6,-\pt) -- (3.4,-\pt) node[black,midway,yshift=-0.5cm,font=\scriptsize] {$q-1$};
\draw [thick,decorate,decoration={brace,amplitude=3pt,mirror,raise=3pt}]
(3.6,-\pt) -- (4.4,-\pt) node[black,midway,yshift=-0.5cm,font=\scriptsize] {$1$};
\draw [thick,decorate,decoration={brace,amplitude=6pt,mirror,raise=3pt}]
(4.6,-\pt) -- (7.4,-\pt) node[black,midway,yshift=-0.5cm,font=\scriptsize] {$p$};
\end{tikzpicture}\ .
\]
The number of $\left(\mybox{n},\mybox{\overline{n}}\right)$ pairs is unchanged, there is one less unpaired $\mybox{n}$ and one more $\mybox{0}$, so
\begin{align*}
\Psi(f_n^\TT R_j) &= (q-p-1) (\beta_{j,n-1}) + (\beta_{j,n})+ 2(p)(\beta_{j,n})\\
&=(q-p-1)(\beta_{j,n-1})+(2p+1)(\beta_{j,n}).
\end{align*}
On the other hand, in $S_n^c\bigl(\Psi(R_j)\bigr)$ the leftmost `(' corresponds to $\beta_{j,n-1}$ so
\[
f_n^{\Kp}\Psi(R_j)= (q-p-1) (\beta_{j,n-1}) + (2p+1)(\beta_{j,j})= \Psi(f_n^\TT R_j).
\]

{\bf Case 3: $\boldsymbol{p \geq q}$, $\boldsymbol{z=1}$, and $\boldsymbol{1\le j<n}$. } 
By the definition of $\Psi$,
\begin{align*}
\Psi(R_j) &= (2q+1)(\beta_{j,n})+(p-q)(\gamma_{j,n}).
\end{align*}
By the definition of $f_n^\TT$, we have
\[
f_n^\TT R_j = 
\begin{tikzpicture}[baseline=-4,xscale=.75,yscale=.65]
\def\pt{.5}
\draw[-] (1-\pt,-\pt) -- (1-\pt,\pt) -- (7+\pt,\pt) -- (7+\pt,-\pt) -- cycle;
\foreach \x in {1.5,2.5,...,6.5}
{\draw[-] (\x,\pt) -- (\x,-\pt);}
\node at (1,0) {$n$};
\node at (2,0) {$\cdots$};
\node at (3,0) {$n$};
\node at (4,0) {$0$};
\node at (5,0) {$\overline{n}$};
\node at (6,0) {$\cdots$};
\node at (7,0) {$\overline{n}$};
\draw [thick,decorate,decoration={brace,amplitude=6pt,mirror,raise=3pt}]
(0.6,-\pt) -- (3.4,-\pt) node[black,midway,yshift=-0.5cm,font=\scriptsize] {$q$};
\draw [thick,decorate,decoration={brace,amplitude=3pt,mirror,raise=3pt}]
(3.6,-\pt) -- (4.4,-\pt) node[black,midway,yshift=-0.5cm,font=\scriptsize] {$0$};
\draw [thick,decorate,decoration={brace,amplitude=6pt,mirror,raise=3pt}]
(4.6,-\pt) -- (7.4,-\pt) node[black,midway,yshift=-0.5cm,font=\scriptsize] {$p+1$};
\end{tikzpicture}\ .
\]
There is one less $\mybox{0}$ and one more unpaired $\mybox{\overline{n}}$, so
\begin{align*}
\Psi(f_n^\TT R_j) &=  2q(\beta_{j,n})+(p-q+1)(\gamma_{j,n}) .
\end{align*}
On the other hand, in $S_n^c\bigl(\Psi(R_j)\bigr)$ the leftmost `(' corresponds to $\beta_{j,n}$ so
\[
f_n^{\Kp}\Psi(R_j)= 2q (\beta_{j,n}) + (p-q+1)(\gamma_{j,n}) = \Psi(f_n^\TT R_j).
\]

{\bf Case 4:  $\boldsymbol{p < q}$, $\boldsymbol{z=1}$ and $\boldsymbol{1\le j<n}$.} 
By the definition of $\Psi$,
\begin{align*}
\Psi(R_j) &= (q-p)(\beta_{j,n-1})+(2p+1)(\beta_{j,n}).
\end{align*}
By the definition of $f_n^\TT$, we have
\[
f_n^\TT R_j = 
\begin{tikzpicture}[baseline=-4,xscale=.75,yscale=.65]
\def\pt{.5}
\draw[-] (1-\pt,-\pt) -- (1-\pt,\pt) -- (7+\pt,\pt) -- (7+\pt,-\pt) -- cycle;
\foreach \x in {1.5,2.5,...,6.5}
{\draw[-] (\x,\pt) -- (\x,-\pt);}
\node at (1,0) {$n$};
\node at (2,0) {$\cdots$};
\node at (3,0) {$n$};
\node at (4,0) {$0$};
\node at (5,0) {$\overline{n}$};
\node at (6,0) {$\cdots$};
\node at (7,0) {$\overline{n}$};
\draw [thick,decorate,decoration={brace,amplitude=6pt,mirror,raise=3pt}]
(0.6,-\pt) -- (3.4,-\pt) node[black,midway,yshift=-0.5cm,font=\scriptsize] {$q$};
\draw [thick,decorate,decoration={brace,amplitude=3pt,mirror,raise=3pt}]
(3.6,-\pt) -- (4.4,-\pt) node[black,midway,yshift=-0.5cm,font=\scriptsize] {$0$};
\draw [thick,decorate,decoration={brace,amplitude=6pt,mirror,raise=3pt}]
(4.6,-\pt) -- (7.4,-\pt) node[black,midway,yshift=-0.5cm,font=\scriptsize] {$p+1$};
\end{tikzpicture}\ .
\]
There is one less $\mybox{0}$, one less unpaired  $\mybox{n}$, and one more $\left( \mybox{n},\mybox{\overline{n}} \right)$ pair, so
\begin{align*}
\Psi(f_n^\TT R_j) &=  (q-p-1)(\beta_{j,n-1})+(2p+2)(\beta_{j,n})
\end{align*}
On the other hand, in $S_n^c\bigl(\Psi(R_j)\bigr)$ the leftmost `(' corresponds to $\beta_{j,n-1}$ so
\[
f_n^{\Kp}\Psi(R_j)= (q-p-1)(\beta_{j,n-1})+(2p+2)(\beta_{j,n})= \Psi(f_n^\TT R_j).
\]

{\bf Case 5:  $\boldsymbol{j=n}$.}
The bracketing sequences are
\begin{align*}
\BR_n(R_n)&=\  )^{2p}\ )^z\ (^z\  (^{2q} , &
S_n\bigl(\Psi(R_n)\bigr)&=\
)^{c_{\beta_{n,n}}}.
\end{align*}
Since there is no `$($' in $S_n\bigl(\Psi(R_n)\bigr)$, $f_n^{\Kp}$ will add $(\beta_{n,n})$ to $\Psi(R_n)$.  

If $z=1$, then the leftmost `$($' in $\BR_n(R_n)$ comes from the $\mybox{0}$ so $f_n^{\TT}(R_i)$ sends the $\mybox{0}$ to $\mybox{\overline{n}}$. According to $\Psi$ we then have that 
\[
\Psi(f_n^{\TT}R_n)=\Psi(R_n)+(\beta_{n,n})=f_n^\Kp\Psi(R_n).
\]
If $z=0$, then the leftmost `$($' comes from an $\mybox{n}$ so $f_n^{\TT}R_n$ sends an $\mybox{n}$ to an $\mybox{0}$.  Again 
\[
\Psi(f_n^{\TT}R_n)=\Psi(R_n)+(\beta_{n,n})=f_n^\Kp\Psi(R_n).
\]

Now, consider $T$ to be of type $C_n$.  Assume row $j$ of $T$ has $p$ $\mybox{\overline{n}}$ boxes and $q$ $\mybox{n}$ boxes,
\[
R_j = 
\begin{tikzpicture}[baseline=-4,xscale=.75,yscale=.65]
\def\pt{.5}
\draw[-] (1-\pt,-\pt) -- (1-\pt,\pt) -- (6+\pt,\pt) -- (6+\pt,-\pt) -- cycle;
\foreach \x in {1.5,2.5,...,5.5}
{\draw[-] (\x,\pt) -- (\x,-\pt);}
\node at (1,0) {$n$};
\node at (2,0) {$\cdots$};
\node at (3,0) {$n$};
\node at (4,0) {$\overline{n}$};
\node at (5,0) {$\cdots$};
\node at (6,0) {$\overline{n}$};
\draw [thick,decorate,decoration={brace,amplitude=6pt,mirror,raise=3pt}]
(0.6,-\pt) -- (3.4,-\pt) node[black,midway,yshift=-0.5cm,font=\scriptsize] {$q$};
\draw [thick,decorate,decoration={brace,amplitude=6pt,mirror,raise=3pt}]
(3.6,-\pt) -- (6.4,-\pt) node[black,midway,yshift=-0.5cm,font=\scriptsize] {$p$};
\end{tikzpicture}\ .
\]
The bracketing sequences are:
\begin{align*}
\BR_n(R_j)&=\    )^p\  (^q\   ,&
S_n\bigl(\Psi(R_j)\bigr)&=\
)^{c_{\gamma_{j,j}}} (^{c_{\beta_{j,n-1}}} )^{c_{\gamma_{j,n}}} (^{c_{\gamma_{j,j}}}.
\end{align*}

{\bf Case 6: $\boldsymbol{p\geq q}$, and $\boldsymbol{1\le j<n}$.} 
By the definition of $\Psi$,
\begin{align*}
\Psi(R_j) &= q (\gamma_{j,j}) + (p-q)(\gamma_{j,n}).
\end{align*}
If $q=0$ then $f_n$ will act on the $\mybox{n}$ in $R_n$ of $T$ (see Case 9  for more details in this situation).  When $q>0$ by the definition of $f_n^\TT$ we have 
\[
f_n^\TT R_j = 
\begin{tikzpicture}[baseline=-4,xscale=.75,yscale=.65]
\def\pt{.5}
\draw[-] (1-\pt,-\pt) -- (1-\pt,\pt) -- (6+\pt,\pt) -- (6+\pt,-\pt) -- cycle;
\foreach \x in {1.5,2.5,...,5.5}
{\draw[-] (\x,\pt) -- (\x,-\pt);}
\node at (1,0) {$n$};
\node at (2,0) {$\cdots$};
\node at (3,0) {$n$};
\node at (4,0) {$\overline{n}$};
\node at (5,0) {$\cdots$};
\node at (6,0) {$\overline{n}$};
\draw [thick,decorate,decoration={brace,amplitude=6pt,mirror,raise=3pt}]
(0.6,-\pt) -- (3.4,-\pt) node[black,midway,yshift=-0.5cm,font=\scriptsize] {$q-1$};
\draw [thick,decorate,decoration={brace,amplitude=6pt,mirror,raise=3pt}]
(3.6,-\pt) -- (6.4,-\pt) node[black,midway,yshift=-0.5cm,font=\scriptsize] {$p+1$};
\end{tikzpicture}\ .
\]
There are two more unpaired $\mybox{\overline{n}}$ and one less $\left(\mybox{n},\mybox{\overline n}\right)$, so
\begin{align*}
\Psi(f_n^\TT R_j) &= (q-1) (\gamma_{j,j}) + (p-q+2)(\gamma_{j,n}).
\end{align*}
On the other hand, in $S_n^c\bigl(\Psi(R_j)\bigr)$ the leftmost `(' corresponds to $\gamma_{j,j}$ so 
\begin{align*}
f_n^{\Kp}\Psi(R_j)&= (q-1) (\gamma_{j,j}) + (p-q+2)(\gamma_{j,n})= \Psi(f_n^\TT R_j).
\end{align*}

{\bf Case 7: $\boldsymbol{q > p+1}$, and $\boldsymbol{1 \le j<n}$.}
By the definition of $\Psi$,
\begin{align*}
\Psi(R_j) &= (q-p)(\beta_{j,n-1})+p (\gamma_{j,j}) .
\end{align*}
By the definition of $f_n^\TT$, we have 
\[
f_n^\TT R_j = 
\begin{tikzpicture}[baseline=-4,xscale=.75,yscale=.65]
\def\pt{.5}
\draw[-] (1-\pt,-\pt) -- (1-\pt,\pt) -- (6+\pt,\pt) -- (6+\pt,-\pt) -- cycle;
\foreach \x in {1.5,2.5,...,5.5}
{\draw[-] (\x,\pt) -- (\x,-\pt);}
\node at (1,0) {$n$};
\node at (2,0) {$\cdots$};
\node at (3,0) {$n$};
\node at (4,0) {$\overline{n}$};
\node at (5,0) {$\cdots$};
\node at (6,0) {$\overline{n}$};
\draw [thick,decorate,decoration={brace,amplitude=6pt,mirror,raise=3pt}]
(0.6,-\pt) -- (3.4,-\pt) node[black,midway,yshift=-0.5cm,font=\scriptsize] {$q-1$};
\draw [thick,decorate,decoration={brace,amplitude=6pt,mirror,raise=3pt}]
(3.6,-\pt) -- (6.4,-\pt) node[black,midway,yshift=-0.5cm,font=\scriptsize] {$p+1$};
\end{tikzpicture}\ .
\]
There is one more $\left(\mybox{n},\mybox{\overline n}\right)$ pair and two less unpaired $\mybox{n}$, so
\begin{align*}
\Psi(f_n^\TT R_j) &= (q-p-2) (\beta_{j,n-1}) + (p+1)(\gamma_{j,j}).
\end{align*}
On the other hand, in $S_n^c\bigl(\Psi(R_j)\bigr)$ the leftmost `(' corresponds to $\beta_{j,n-1}$ so
\[
f_n^{\Kp}\Psi(R_j)= (q-p-2) (\beta_{j,n-1}) + (p+1)(\gamma_{j,j})= \Psi(f_n^\TT R_j).
\]

{\bf Case 8: $\boldsymbol{q = p+1}$, and $\boldsymbol{j<n}$.}
By the definition of $\Psi$,
\begin{align*}
\Psi(R_j) &= (q-p)(\beta_{j,n-1})+p (\gamma_{j,j}) .
\end{align*}
By the definition of $f_n^\TT$, we have 
\[
f_n^\TT R_j = 
\begin{tikzpicture}[baseline=-4,xscale=.75,yscale=.65]
\def\pt{.5}
\draw[-] (1-\pt,-\pt) -- (1-\pt,\pt) -- (6+\pt,\pt) -- (6+\pt,-\pt) -- cycle;
\foreach \x in {1.5,2.5,...,5.5}
{\draw[-] (\x,\pt) -- (\x,-\pt);}
\node at (1,0) {$n$};
\node at (2,0) {$\cdots$};
\node at (3,0) {$n$};
\node at (4,0) {$\overline{n}$};
\node at (5,0) {$\cdots$};
\node at (6,0) {$\overline{n}$};
\draw [thick,decorate,decoration={brace,amplitude=6pt,mirror,raise=3pt}]
(0.6,-\pt) -- (3.4,-\pt) node[black,midway,yshift=-0.5cm,font=\scriptsize] {$q-1$};
\draw [thick,decorate,decoration={brace,amplitude=6pt,mirror,raise=3pt}]
(3.6,-\pt) -- (6.4,-\pt) node[black,midway,yshift=-0.5cm,font=\scriptsize] {$p+1$};
\end{tikzpicture}\ .
\]
Since $q-p=1$ the number of $\left(\mybox{n},\mybox{\overline n}\right)$ pairs is unchanged. There is one less $\mybox{n}$ and one more $\mybox{\overline{n}}$, so
\begin{align*}
\Psi(f_n^\TT R_j) &= (q-p-1) (\beta_{j,n-1}) + (\gamma_{j,n})+p(\gamma_{j,j}).
\end{align*}
On the other hand, in $S_n^c\bigl(\Psi(R_j)\bigr)$ the leftmost `(' corresponds to $\beta_{j,n-1}$ so
\[
f_n^{\Kp}\Psi(R_j)= (q-p-1) (\beta_{j,n-1})+ (\gamma_{j,n}) + p(\gamma_{j,j})= \Psi(f_n^\TT R_j).
\]

{\bf Case 9: $\boldsymbol{j=n}$.} 
The only positive root that can be in $\Psi(R_n)$ is $(\gamma_{n,n})$, so there is no `$($' in $S_n\bigl(\Psi(R_n)\bigr)$ and, by Definition \ref{def:KPops}, $f_n^{\Kp}$ adds a $(\gamma_{n,n})$. The leftmost `$($' in $\BR_n(T)$ comes from an $\mybox{n}$, so $f_n^{\TT}$ sends an $\mybox{n}$ to an $\mybox{\overline{n}}$.  Hence $\Psi(f_n^{\TT}R_n)=\Psi(R_n)+(\gamma_{n,n})=f_n^\Kp\Psi(R_n)$.
\end{proof}

\begin{proof}[Proof of Theorem \ref{isom_BC}] It suffices to show that for all $i$ we have $f_i^\Kp\Psi(T) = \Psi(f_i^{\TT}T)$.
By the definition of the bracketing sequences and of $\Psi$, we have
\begin{equation*}
\BR_i(T) \text{  factors as }
\BR_i(R_1) \BR_i(R_2) \cdots \BR_i(R_{n}), \text{ and}
\end{equation*}
\begin{equation*} S_i\bigl(\Psi(T)\bigr) \text{ factors as }
S_i\bigl(\Psi(R_1)\bigr) S_i\bigl(\Psi(R_2)\bigr) \cdots S_i\bigl(\Psi(R_{n})\bigr).
\end{equation*}
\noindent Suppose that the leftmost `$($' in $\BR_i^c(T)$ comes from row $R_j$. There will always be an uncanceled bracket coming from row $i$ so we may assume $j\leq i$.
By applying Lemma \ref{lem:onerow_i} or Lemma \ref{lem:onerow_n} to each $R_j$, the leftmost `(' in $S_i\bigl(\Psi(T)\bigr)$ comes from $S_i\bigl(\Psi(R_j)\bigr)$, and also $\Psi \bigl(f_i^\TT R_j\bigr) = f_i^{\Kp} \Psi(R_j)$. The result follows. 
\end{proof}

\section{Stack notation}
This work extends results from \cite{CT15, SST2} in types $A_n$ and $D_n$ to types $B_n$ and $C_n$. The type $A_n$ result can be described using the multisegments from \cite{JL09,LTV99,Z80} which are a diagrammatic notation that makes the crystal structure apparent.  In \cite{SST2} this was extended to type $D_n$ by introducing a \emph{stack} notation for Kostant partitions in which the crystal structure can easily be read off. We now define a  similar \emph{stack} notation for types $B_n$ and $C_n$. 

In type $B_n$ we associate positive roots to ``stacks" with 
\[ 
\beta_{j,k} = \stack{k\\\mdots\\j}\ , \qquad 
\gamma_{\ell,m} = \stack{ m \\ \mdots \\ n-1 \\  n \ n \\ n-1 \\ \mdots \\ \ell }\ ,
\]
for $1\leq j \leq k \leq n$ and $1\leq \ell < m \leq n$.

In type $C_n$ we associate positive roots to ``stacks" with
\[
\beta_{j,k} = \stack{k\\\mdots\\j}\ , \qquad 
\gamma_{\ell,m} = \stack{ m \\ \mdots \\ n-1 \\  n \\ n-1 \\ \mdots \\ \ell }\ , \qquad
\gamma_{h,h} = \stack{ n \\ n-1 \ n-1 \\ \mdots \\ h \ h }\ ,
\]
for $1\leq j \leq k < n$, $1\leq \ell < m \leq n$, and $ 1\leq h \leq n$.

Then the sequences of roots $\Phi_i$ from Definition \ref{def:KPisignature} are exactly those positive roots where we can either add or remove an $i$ from the top  of the corresponding stack and still have either a valid stack, an empty stack, or in type $C_n$ with $i=n$ where we have two valid stacks side by side. Once the stacks are ordered as in Definition \ref{def:KPisignature}, the bracketing sequence is created by placing a left bracket for each $i$ that can be added to the top of a stack, and a right bracket for each $i$ that can be removed from the top.  Note that if both happen then the root corresponding to the stack appears twice in Definition \ref{def:KPisignature}, in which case the `)' is placed over the left copy and the `(' over the right copy. 
If there is a leftmost uncanceled `$($' the crystal operator $f_i$ adds an $i$ to the top of the corresponding stack (or, in the case of $i=n$ in type $C_n$, may combine two stacks together and attach an $n$ at the top).  Otherwise $f_i$ creates a new stack consisting of just $i$. 
%

\begin{Remark}
Being able to add or remove an $i$ from the top of a stack is different from being able to add or remove an $\alpha_i$ from the corresponding root. For instance, in type $B_3$, if $\beta= \alpha_1+\alpha_2+2\alpha_3$, then $\beta-\alpha_1$ is a root, but there is no 1 at the top of the stack corresponding to $\beta$, so $\beta$ is not in $\Phi_1^B$. Similarly, in type $C_3$, although
$ \stack{ 2 \\ 3 \\ 2 \\ 1 }$ is a stack, $\alpha_1+2\alpha_2+\alpha_3$ is not in $\Phi_1^C$ because the stack for $2\alpha_1+2\alpha_2+\alpha_3$ is $ \stack{  3 \\ 2\ 2 \\ 1\ 1}$, not $ \stack{  1\\ 2 \\ 3 \\ 2 \\ 1 }$. 
\end{Remark}

\begin{Example}\label{ex:StackCnormal}
Consider type $C_3$ and $\bm\alpha\in \Kp(\infty)$ given in stack notation by
	\[
	\bm\alpha = \stack{2 \\ 1} \ \stack{2 \\ 1} \ \stack{2 \\ 1} \ \stack{3 \\ 2 \\ 1} \ \stack{3 \\ 2 \\ 1} \ \stack{3 \\2 \ 2 \\ 1 \ 1} \ \stack{3 \\2 \ 2 \\ 1 \ 1} \ \stack{3 \\2\ 2} \ \stack{2 \\ 3} \ \stack{3}.
	\]
	The corresponding $3$-signature is 
	\[
	\begin{array}{cccccccccccccccc}
	&  \stack{3\\2\ 2\\1\ 1} &  \stack{3\\2\ 2\\1\ 1}  & \stack{2 \\ 1 } & \stack{2 \\ 1 } & \stack{2 \\ 1 } & \stack{3 \\ 2 \\ 1} & \stack{3 \\ 2 \\ 1}&  \stack{3\\2\ 2\\1\ 1} & \stack{3\\2\ 2\\1\ 1} & \stack{3\\2\ 2} & \stack{3 \\ 2} & \stack{3\\2\ 2} & \stack{3}\\
	S_3(\bm\alpha)= & ) & ) & ( &\color{red}{(}&\color{red}{(} & \color{red}{)} &\color{red}{)} & \color{red}{(} & \color{red}{(} & \color{red}{)} & \color{red}{)} & \color{red}{(} & \color{red}{)} \\ 
	S_3^c(\bm\alpha)=& )& )& \color{blue}{(} & & & & & & & & & & &.
	\end{array}
	\] 
	Thus the action of $f_3$ on $\bm\alpha$ adds a 3 to top of a $\stack{2 \\ 1}$. This gives
	\[
	f_3\bm\alpha = \stack{2 \\ 1} \ \stack{2 \\ 1} \ \stack{3 \\ 2 \\ 1} \ \stack{3 \\ 2 \\ 1} \ \stack{3 \\ 2 \\ 1} \ \stack{3 \\2 \ 2 \\ 1 \ 1} \ \stack{3 \\2 \ 2 \\ 1 \ 1} \ \stack{3 \\2\ 2} \ \stack{2 \\ 3} \ \stack{3}.
	\]
\end{Example}

\begin{Example}\label{ex:StackCsplit}
Consider type $C_3$ and $\bm\alpha$ as in Example \ref{ex:KPopsC_3f2e}. In stack notation, 
\[
\bm\alpha = \stack{2 \\ 1} \ \stack{2 \\ 1} \ \stack{3 \\ 2 \\ 1} \ \stack{3 \\ 2 \\ 1} \ \stack{3 \\2 \ 2 \\ 1 \ 1} \ \stack{3 \\2 \ 2 \\ 1 \ 1} \ \stack{3 \\2 \ 2 \\ 1 \ 1} \ \stack{3 \\2\ 2} \ \stack{3 \\ 2} \ \stack{3}.
\]
Recalculating the $3$-signature using stack notation gives
\[
\begin{array}{ccccccccccccccccc}
&  \stack{3 \\2 \ 2 \\ 1 \ 1} &  \stack{3 \\2 \ 2 \\ 1 \ 1} &  \stack{3 \\2 \ 2 \\ 1 \ 1} & \stack{2 \\ 1 } & \stack{2 \\ 1 } & \stack{3 \\ 2 \\ 1} & \stack{3 \\ 2 \\ 1}&  \stack{3 \\2 \ 2 \\ 1 \ 1}& \stack{3 \\2 \ 2 \\ 1 \ 1} & \stack{3 \\2 \ 2 \\ 1 \ 1}& \stack{3 \\2\ 2} & \stack{3 \\ 2} & \stack{3 \\2\ 2} & \stack{3}\\[10pt]
S_3(\bm\alpha)= & ) & ) & ) &\color{red}{(}&\color{red}{(} & \color{red}{)} &\color{red}{)} & ( & \color{red}{(} & \color{red}{(} & \color{red}{)} & \color{red}{)} & \color{red}{(} & \color{red}{)} \\ 
S_3^c(\bm\alpha)=& )& )& ) & & & &  &\color{blue}{(} & & & & & & &.
\end{array}
\] 
Since the leftmost `$($' comes from a  $ \stack{3 \\2 \ 2 \\ 1 \ 1}$, we should add a $3$ to the top of this stack, which gives $ \stack{3 \ 3\\2 \ 2 \\ 1 \ 1}$. That is not the stack of a single root, but should be thought of as two copies of $ \stack{3 \\ 2 \\ 1}$, which is the stack of a root. The result is
\[
f_3\bm\alpha = \stack{2 \\ 1} \ \stack{2 \\ 1} \ \stack{3 \\ 2 \\ 1} \ \stack{3 \\ 2 \\ 1} \  \stack{3 \\ 2 \\ 1} \  \stack{3 \\ 2 \\ 1} \ \stack{3 \\2 \ 2 \\ 1 \ 1} \ \stack{3 \\2 \ 2 \\ 1 \ 1}\ \stack{3 \\2\ 2} \ \stack{3 \\ 2} \ \stack{3}\ .
\]
\end{Example}

\bibliographystyle{amsplain}
\bibliography{KP-crystal}{}

\end{document}